\documentclass[11pt]{amsart}
\usepackage{amsmath,amssymb,amsthm,mathtools}
\usepackage[colorlinks=true,linkcolor=blue,citecolor=blue,urlcolor=blue]{hyperref}

\theoremstyle{plain}
\newtheorem{theorem}{Theorem}[section]
\newtheorem{lemma}[theorem]{Lemma}
\newtheorem{proposition}[theorem]{Proposition}
\newtheorem{corollary}[theorem]{Corollary}
\theoremstyle{definition}
\newtheorem{definition}[theorem]{Definition}
\newtheorem{example}[theorem]{Example}
\newcommand{\Fp}{\overline{\mathbf F}_{p}}
\newcommand{\PP}{\mathbf P}
\newcommand{\length}{\ell}
\numberwithin{equation}{section}

\title[Hurwitz existence in prime degree]
{The Hurwitz existence problem in prime degree}
\author{Jijian Song, Hailin Wen, Zebao Zhang}
\date{\today}
\subjclass[2020]{Primary 14H30; Secondary 11G20, 14N35, 20B25, 30F10}
\keywords{branched cover, Hurwitz existence problem, prime degree,
stable reduction, deformation datum, Hurwitz space, relative
Gromov--Witten invariant}
\hypersetup{
  pdftitle={The Hurwitz existence problem in prime degree}
}

\begin{document}
\begin{abstract}
Let $p$ be a prime.  We prove that every compatible branch datum of degree
$p$ over the sphere is realizable by a connected branched cover.  The
three-point case is constructed in residue characteristic $p$.  Henrio's
moment theorem supplies the distinct-point moment solutions from which we
construct a special primitive tail for each prescribed partition; a second
application underlies the new tail required by a positive source genus.
These tails are joined by a logarithmic deformation datum and embedded in
one subgroup of $S_p$ containing a common regular subgroup of order $p$.
Wewers's lifting theorem produces a three-point Galois cover in
characteristic zero.  The quotient by a point stabilizer has degree $p$ and
the prescribed three ramification profiles.  The fusion and realization
results of Edmonds--Kulkarni--Stong then give the assertion for an arbitrary
number of branch values.  As consequences, the connected prime-degree
Hurwitz potential has full support on the Riemann--Hurwitz locus, every
corresponding connected relative Gromov--Witten invariant of $\mathbf P^1$
is nonzero, the two-relative-point disconnected sector with even
completed-cycle orders at most $p$ is strictly positive subject to the
dimension constraint, and the connected transposition sector is strictly
positive.
\end{abstract}
\maketitle

\section{Introduction}

For a partition $\Lambda=(e_1,\ldots,e_r)\vdash d$, define its defect by
\[
 v(\Lambda)=d-r=\sum_{j=1}^{r}(e_j-1).
\]
A branch datum of degree $d$ over the sphere is a finite ordered family of
nonidentity partitions $\Lambda_1,\ldots,\Lambda_k\vdash d$.  It is
\emph{compatible} with an orientable source of genus $g$ if
\begin{equation}\label{eq:RH-intro}
 \sum_{i=1}^{k}v(\Lambda_i)=2d-2+2g.
\end{equation}
Indeed, the Riemann--Hurwitz formula reads
$2-2g=2d-\sum_{i,j}(e_{ij}-1)$, and the inner sum is
$v(\Lambda_i)$.

\begin{theorem}\label{thm:main}
Let $p$ be a prime.  Every branch datum of degree $p$ satisfying
\eqref{eq:RH-intro} is realized by a connected branched cover
$\Sigma_g\to S^2$.
\end{theorem}

The assertion was formulated by Edmonds--Kulkarni--Stong
\cite[p.~787]{EKS}.  Chenakkod--Faraco--Gupta
\cite[Corollary~9.6]{ChenakkodFaracoGupta} prove the case in which all but
one profile are pure cycles.  Ciliberto--Knutsen--Torelli
\cite[Theorem~0.1]{CilibertoKnutsenTorelli} give an alternative proof of the
subcase in which every profile is pure-cycle, with prescribed branch points.
Further partial results and structural criteria appear in
\cite{PascaliPetronio,BaroniPetronio,PakovichFiber,WeiRational}.
Wang et al. give a complete nonredundant enumeration of non-realizable
partition triples through degree $31$ and thereby verify the prime-degree
conjecture for primes below $32$ \cite{WangComputational}.  As a dated record
of the state of the problem, Meng--Wei--Zhou described the prime-degree
conjecture as unresolved in May 2026 \cite[Introduction]{MengWeiZhou}.  A
subsequent criterion of Wei treats a specified composite-degree
three-profile family \cite{WeiThreeBranch}.

The proof uses stable reduction in residue characteristic $p$.  Henrio's
moment criterion
\cite[Definition~3.14 and Proposition~3.16(b)]{Henrio} supplies pairwise
distinct points satisfying the required moment equations.  We use these
points to construct primitive tails with arbitrary prescribed tame profiles.
We assemble these tails into a special $G$-deformation datum in Wewers's
sense and then apply Wewers's lifting theorem
\cite{WewersMeta,WewersThree} to obtain a three-point cover in
characteristic zero.  Bouw
\cite[Section~3, especially the construction preceding Lemma~3.2 and
Lemma~3.2]{BouwAffine} constructs primitive tails with a prescribed finite
tame profile in the range in which its length is at most $(p+1)/2$.
Common-$p$-cycle assemblies followed by a point-stabilizer quotient occur
for pure-cycle genus-zero data in
\cite[Construction~4.2.3 and Propositions~4.2.4--4.2.5]{Eskin}.
Prescribed-type deformation data and the corresponding local lifting
problem are treated in \cite{BouwWewersZapponi}; related four-point
calculations appear in \cite{BouwOsserman}.  For $g=1$, the invariant
$1+2/(p-1)$ of the one-point construction below is the invariant in
\cite[Proposition~3.3]{BouwAffine}; Bouw proves that proposition by a
quadratic base change.

After proving Theorem~\ref{thm:main}, we give two moduli-theoretic
consequences.  First, every compatible prime-degree passport determines a
nonzero Hurwitz cycle with full support over the configuration space of its
branch values.  Second, the Gromov--Witten/Hurwitz correspondence of
Okounkov--Pandharipande \cite{OkounkovPandharipande} converts the main
theorem into positivity statements for relative stationary invariants of
$\mathbf P^1$.  With no stationary insertions, their theorem identifies the
connected relative virtual degree with the connected Hurwitz number for
every passport.  With insertions, the nonnegative nonconstant coefficients
of even completed cycles give a strictly positive two-relative-point
disconnected sector when their orders are at most $p$ and the dimension
constraint holds.  In the transposition direction, where no completion
correction occurs, taking the logarithm gives a strictly positive connected
sector.  The ordinary
fundamental cycle of a Hurwitz stack and the virtual class of the relative
stable-map compactification are kept distinct.

\section{Branched covers and permutation data}

\begin{definition}
Let $X$ and $B$ be closed connected oriented surfaces.  An
orientation-preserving branched cover $f:X\to B$ of degree $d$ is a
continuous surjection for which every point $x\in X$ has oriented local
coordinates in which $f$ is $z\mapsto z^{e_x}$ for an integer $e_x\geq1$.
The integer $e_x$ is the local degree.  If
$f^{-1}(b)=\{x_1,\ldots,x_r\}$, then
\[
 \operatorname{prof}_b(f)=(e_{x_1},\ldots,e_{x_r})\vdash d
\]
is the ramification profile above $b$.  A branch value is a point for which
at least one $e_{x_j}$ exceeds $1$.  A branch datum is realized if its
entries, with their multiplicities, are the profiles over the branch values
of such a cover with connected source.
\end{definition}

Let $b_1,\ldots,b_k$ be distinct points of $\mathbf P^1(\mathbf C)$, put
$U=\mathbf P^1-\{b_1,\ldots,b_k\}$, and choose positively oriented
peripheral loops $\gamma_i$ with $\gamma_1\cdots\gamma_k=1$.

\begin{theorem}[Hurwitz monodromy criterion]\label{thm:hurwitz-criterion}
Let $\Lambda_i\vdash d$.  There exists a connected degree-$d$ branched
cover of $\mathbf P^1(\mathbf C)$ with branch values $b_i$ and profiles
$\Lambda_i$ if and only if there are permutations
$\sigma_i\in S_d$ such that
\[
 \operatorname{type}(\sigma_i)=\Lambda_i,\qquad
 \sigma_1\cdots\sigma_k=1,
\]
and $\langle\sigma_1,\ldots,\sigma_k\rangle$ is transitive on
$\{1,\ldots,d\}$.
\end{theorem}

\begin{proof}
Fix a point $u\in U$ and label its fiber by $\{1,\ldots,d\}$.  Path lifting
defines a homomorphism $\rho:\pi_1(U,u)\to S_d$.  A connected covering has
transitive monodromy: two labels lie in the same orbit precisely when their
points in the fiber lie in the same connected component of the covering
space.  The presentation
\[
 \pi_1(U,u)=
 \langle\gamma_1,\ldots,\gamma_k\mid
   \gamma_1\cdots\gamma_k=1\rangle
\]
gives $\rho(\gamma_1)\cdots\rho(\gamma_k)=1$.  If a cycle of
$\rho(\gamma_i)$ has length $e$, the corresponding component above a small
punctured disk about $b_i$ is the connected degree-$e$ cover of that
punctured disk.  Completing it by one point gives the local model
$z\mapsto z^e$.  Hence the cycle lengths are the local degrees.

Conversely, permutations satisfying the product relation define a
homomorphism from the displayed presentation to $S_d$.  The standard
covering-space construction gives a degree-$d$ covering of $U$, and
transitivity makes it connected.  For every cycle of length $e$ in
$\sigma_i$, attach a disk to the associated boundary component by the map
$z\mapsto z^e$.  The result is a closed oriented surface and an
orientation-preserving branched cover with the required profiles.  This is
the form of the Hurwitz criterion used in \cite[Lemma~2.1]{EKS}.
\end{proof}

\begin{lemma}[Profiles in an intermediate quotient]
\label{lem:quotient-profile}
Let $Y\to B$ be a connected $G$-Galois cover of smooth curves over an
algebraically closed field, let $H\leq G$, and put $X=Y/H$.  Fix a point
$y\in Y$ and let $I\leq G$ be its inertia group.  The points of $X$ above
the image of $y$ correspond to the $I$-orbits on $G/H$.  The local degree at
the point corresponding to the orbit of $gH$ is
\[
 [I:I\cap gHg^{-1}],
\]
which is the cardinality of that orbit.  If the ramification is tame, then
$I$ is cyclic and the profile of $X\to B$ is the cycle partition of any
generator of $I$ acting on $G/H$.
\end{lemma}

\begin{proof}
The points of $X$ above the chosen point of $B$ are indexed by the double
cosets $I\backslash G/H$, equivalently by the $I$-orbits on $G/H$.  The
stabilizer in $I$ of $gH$ is
\[
 \{i\in I:igH=gH\}=I\cap gHg^{-1}.
\]
The orbit--stabilizer formula gives the displayed index.  The same index is
the ramification index in the corresponding intermediate extension of
complete local fields.  In the tame case $I$ is cyclic.  A generator has
one cycle on each $I$-orbit and the length of that cycle is the orbit
cardinality, so its cycle partition is the ramification profile.
\end{proof}

\begin{lemma}[Normal closure over an \'etale locus]
\label{lem:normal-closure-etale}
Let $U$ be a connected normal curve over a field and let $V\to U$ be a
connected finite \'etale cover.  If $L$ is the normal closure of the finite
extension $k(V)/k(U)$, then the normalization of $U$ in $L$ is finite
\'etale over $U$.
\end{lemma}

\begin{proof}
Finite \'etale covers of $U$ are closed under fiber products, open-and-closed
subschemes, and connected components.  Choose the finitely many
$k(U)$-embeddings of $k(V)$ into a fixed separable closure of $k(U)$.  The
compositum of their images is $L$.  The normalization of $U$ in this
compositum is a connected component of an iterated fiber product of copies
of $V$ over $U$.  That iterated fiber product is finite \'etale over $U$, and
so is each connected component.  Hence the normalization in $L$ is finite
\'etale over $U$.
\end{proof}

\begin{lemma}[The normalizer of a regular subgroup]\label{lem:normalizer}
Let $p$ be prime and let $P\leq S_p$ be generated by a $p$-cycle.  Then
\[
 C_{S_p}(P)=P,\qquad
 N_{S_p}(P)=P\rtimes\mathbf F_p^\times.
\]
For every divisor $m$ of $p-1$, the quotient
$N_{S_p}(P)/P$ has a unique subgroup of order $m$, and all complements of
$P$ in its inverse image are conjugate by an element of $P$.
\end{lemma}

\begin{proof}
Identify the set of letters with the additive group of $\mathbf F_p$.  For
$b\in\mathbf F_p$, write $\tau_b(x)=x+b$; then
$P=\{\tau_b:b\in\mathbf F_p\}$.  If a permutation $f$ centralizes $P$ and
$c=f(0)$, then
\[
 f(b)=f(\tau_b(0))=\tau_b(f(0))=b+c
\]
for every $b$.  Hence $f=\tau_c\in P$, which proves $C_{S_p}(P)=P$.

Suppose next that $f$ normalizes $P$.  Conjugation by $f$ induces an
automorphism of the cyclic group $P$.  There is therefore a unique
$a\in\mathbf F_p^\times$ such that
$f\tau_bf^{-1}=\tau_{ab}$ for all $b$.  If $c=f(0)$, then
\[
 f(b)=f(\tau_b(0))=(f\tau_bf^{-1})(c)=ab+c.
\]
Thus $f$ is affine.  Conversely, every map $x\mapsto ax+c$ with
$a\ne0$ normalizes $P$, so
$N_{S_p}(P)=P\rtimes\mathbf F_p^\times$.

The quotient $N_{S_p}(P)/P$ is the cyclic group $\mathbf F_p^\times$.
Consequently it has a unique subgroup $A_m$ of order $m$ for every
$m\mid p-1$.  Let $H$ be a complement to $P$ in $P\rtimes A_m$.  Projection
to $A_m$ restricts to an isomorphism $H\to A_m$.  If $m=1$, then $H$ is
trivial.  If $m>1$, choose a generator of $H$ and write it as
$x\mapsto ax+c$, where $a$ has order $m$ and hence $a\ne1$.  It fixes the
unique point $x_0=c/(1-a)$.  Conjugating by the translation
$x\mapsto x-x_0$ carries this generator to $x\mapsto ax$ and carries $H$ to
the standard complement $A_m$.  Hence any two complements are conjugate by
an element of $P$.
\end{proof}

\section{Reduction to three branch values}\label{sec:EKS-reduction}

Let $p$ be odd, and let $\Lambda_i\vdash p$ be nonidentity partitions,
none equal to $[p]$.  Write $r_i=\length(\Lambda_i)$.  For three branch
values, compatibility is
\begin{equation}\label{eq:three-RH}
 r_1+r_2+r_3=p+2-2g.
\end{equation}
Each $2\leq r_i\leq p-1$.  Since the left side of
\eqref{eq:three-RH} is odd, it is at least $7$, and therefore
\begin{equation}\label{eq:genus-bound}
 0\leq g\leq \frac{p-5}{2}.
\end{equation}

We use the following five results of Edmonds--Kulkarni--Stong.  They are
stated in the form needed below.  If $A$ is a partition, its defect is
denoted by $v(A)$, and a permutation belongs to $A$ when its cycle partition
is $A$.

\begin{proposition}[EKS fusion and realization results]
\label{prop:EKS-inputs}
Let $p$ be an odd integer and let $A,B$ be partitions of $p$.
\begin{enumerate}
\item If $v(A)+v(B)\leq p-1$, there are $\alpha\in A$ and
$\beta\in B$ such that $v(\alpha\beta)=v(A)+v(B)$.
\item If $v(A)+v(B)$ is even and at least $p+1$, there are
$\alpha\in A$ and $\beta\in B$ such that $\alpha\beta$ is a $p$-cycle.
\item If $v(A)+v(B)$ is odd and at least $p$, there are
$\alpha\in A$ and $\beta\in B$ such that $\alpha\beta$ has cycle
partition $[p-1,1]$.
\item Every compatible degree-$p$ branch datum containing $[p]$ is
realizable.
\item Every compatible degree-$p$ branch datum containing $[p-1,1]$ is
realizable.
\end{enumerate}
\end{proposition}

\begin{proof}
Part (1) is Lemma~4.2 on pp.~781--782 of \cite{EKS}, part (2) is
Corollary~4.4 on p.~782, and part (3) is Lemma~4.5 on pp.~782--783.  In
Lemma~4.5 the alternative partition
$[p/2,p/2]$ occurs only in even degree and is therefore absent here.
Parts (4) and (5) are Propositions~5.2 and~5.3 on pp.~784--785; all
exceptional cases in Proposition~5.3 have even degree.
\end{proof}

\begin{lemma}[Insertion of a factorization]
\label{lem:factor-insertion}
Let $c_1,\ldots,c_n\in S_p$ satisfy
\[
 c_1\cdots c_n=1,
 \qquad
 \langle c_1,\ldots,c_n\rangle\ \text{is transitive}.
\]
Fix an index $i$.  Let
$\alpha,\beta\in S_p$ and suppose that $\alpha\beta$ is conjugate to $c_i$.
Then there is $u\in S_p$ for which replacing $c_i$ by the two consecutive
factors $u\alpha u^{-1}$ and $u\beta u^{-1}$ produces a product-one tuple
that still generates a transitive subgroup.
\end{lemma}

\begin{proof}
Choose $u\in S_p$ such that $u\alpha\beta u^{-1}=c_i$.  The product of the
two inserted factors is
\[
 (u\alpha u^{-1})(u\beta u^{-1})
   =u\alpha\beta u^{-1}=c_i.
\]
Consequently the product of the expanded tuple equals the product of the
original tuple and is therefore $1$.  Let $H'$ be the subgroup generated by
the expanded tuple.  It contains every $c_j$ with $j\ne i$, and it contains
$c_i$ as the product of the two inserted factors.  Hence
$\langle c_1,\ldots,c_n\rangle\leq H'$.  A group containing a transitive
subgroup is transitive, so $H'$ is transitive.
\end{proof}

\begin{lemma}[EKS interface]\label{lem:EKS}
Suppose that every compatible triple as above is realizable.  Then every
compatible branch datum of odd prime degree is realizable.
\end{lemma}

\begin{proof}
We reconstruct the reduction in \cite[Sections 4--5]{EKS}, including its
connectedness argument.  We use strong induction on the number $k$ of
genuine branch partitions.  Compatibility is impossible for $k\leq1$, since
every partition has branching at most $p-1<2p-2$.  If $k=2$, compatibility
forces $g=0$ and both entries to be $[p]$; inverse $p$-cycles realize them.
If $k=3$, Proposition \ref{prop:EKS-inputs}(4) applies when one entry is
$[p]$, and otherwise the residual three-point hypothesis applies.  More
generally, Proposition \ref{prop:EKS-inputs}(4) realizes any compatible datum
containing $[p]$.  We may therefore assume
$k\geq4$ and that no entry is $[p]$.

Suppose first that two entries $A,B$ satisfy
\[
 s:=v(A)+v(B)\leq p-1.
\]
Proposition \ref{prop:EKS-inputs}(1) gives $\alpha\in A$, $\beta\in B$ such that
$v(\alpha\beta)=s$.  Replace $A,B$ by the type $C$ of $\alpha\beta$.
Because $A$ and $B$ are nonidentity partitions, $v(A)\geq1$ and
$v(B)\geq1$.  Hence $v(C)=s\geq2$, so $C$ is also a nonidentity partition.
The new datum has $k-1$ entries and exactly the same total branching, hence
the same genus.  By induction it has a transitive product-one tuple.  Lemma
\ref{lem:factor-insertion} replaces its $C$-element by conjugates of
$\alpha,\beta$ without changing the product-one relation or transitivity.
Thus the expanded tuple realizes the original datum connectedly.

It remains to consider $k\geq4$ when every pair satisfies
$v(A_i)+v(A_j)\geq p$.  Choose $A,B$, put
\[
 s=v(A)+v(B),\qquad
 R=\sum_{C\ne A,B}v(C).
\]
At least two entries remain, so the pairwise hypothesis gives $R\geq p$.
If $s$ is even, then it has the parity of $p+1$ and is at least $p+1$.
Proposition \ref{prop:EKS-inputs}(2) chooses $\alpha,\beta$ with product a
$p$-cycle.  Replace
$A,B$ by $[p]$.  Here $R$ is even; since $p$ is odd, $R\geq p+1$, and the
reduced total branching is
\[
 R+(p-1)\geq2p.
\]
Define the temporary source genus by
\[
 g'=\frac{R+(p-1)-(2p-2)}2=\frac{R-p+1}{2}\geq1.
\]
The numerator is even because $R$ is even and $p$ is odd.  Thus
$g'\in\mathbf Z_{\geq1}$ and the reduced total branching is
$2p-2+2g'$.  The reduced datum is compatible and contains $[p]$, so
Proposition \ref{prop:EKS-inputs}(4) realizes it.

If $s$ is odd, then $s\geq p$ and $s-p$ is even.  Put
$q=(s-p)/2\in\mathbf Z_{\geq0}$, so that $s=p+2q$.  Proposition
\ref{prop:EKS-inputs}(3) chooses $\alpha,\beta$ whose product has type
$[p-1,1]$.  Now $R$ is odd, hence $R\geq p$,
and the reduced total branching is
\[
 R+(p-2)\geq2p-2.
\]
Define the temporary genus by
\[
 g'=\frac{R+(p-2)-(2p-2)}2=\frac{R-p}{2}\geq0.
\]
Here $R-p$ is even because both $R$ and $p$ are odd.  Thus
$g'\in\mathbf Z_{\geq0}$ and the reduced total branching is
$2p-2+2g'$.  The reduced datum is compatible and contains $[p-1,1]$, so
Proposition \ref{prop:EKS-inputs}(5) realizes it.
In each high-pair case, Lemma \ref{lem:factor-insertion} expands the selected
product and restores the original profiles, product-one relation, and
transitivity.  Although a high-pair
replacement changes the intermediate total branching and genus, the final
expanded tuple has the original total and therefore the prescribed genus.

Identity partitions $1^p$ were excluded from branch data at the outset.  If
an externally marked list contains one, deleting it changes neither the
product, the generated subgroup, nor Riemann--Hurwitz, and reinserting it adds
only an unramified marked target point.  This proves the lemma.
\end{proof}

The case $p=2$ will be treated at the end.  We now prove the residual
three-point assertion uniformly for odd $p$.

\section{Deformation data, tail covers, and lifting}

Throughout the characteristic-$p$ construction in the next five sections,
$k$ is an algebraically closed field of characteristic $p>0$.  In the
applications of Henrio's theorem we take $k=\overline{\mathbf F}_p$.

\begin{definition}[Deformation datum]\label{def:deformation-datum}
Let $X$ be a smooth projective connected curve over $k$.  A deformation
datum of type $(H,\chi)$ over $X$ is a pair $(Z,\omega)$ consisting of a
connected tame $H$-Galois cover $Z\to X$, with $p\nmid |H|$, and a nonzero
meromorphic differential $\omega$ on $Z$ such that
\[
 h^*\omega=\chi(h)\omega\qquad(h\in H),
\]
where $\chi:H\to\mathbf F_p^\times$ is a character.  The datum is
logarithmic if $\omega=du/u$ for a rational function $u\in k(Z)^\times$.
Only logarithmic deformation data will be used below.

For $\tau\in X$, choose $\xi\in Z$ above $\tau$ and set
\[
 m_\tau=|H_\xi|,\qquad
 h_\tau=\operatorname{ord}_\xi(\omega)+1,\qquad
 \sigma_\tau=\frac{h_\tau}{m_\tau}.
\]
The point $\tau$ is critical if $(m_\tau,h_\tau)\ne(1,1)$.
\end{definition}

\begin{lemma}[Independence of the point above $\tau$]
\label{lem:deformation-invariants}
The integers $m_\tau,h_\tau$ and the rational number $\sigma_\tau$ in
Definition \ref{def:deformation-datum} do not depend on the chosen point
$\xi$ above $\tau$.
\end{lemma}

\begin{proof}
Let $\xi'$ be another point above $\tau$.  Since $Z\to X$ is a connected
$H$-Galois cover, $H$ acts transitively on the geometric fiber above
$\tau$.  Choose $a\in H$ with $a(\xi)=\xi'$.  Then
\[
 H_{\xi'}=aH_\xi a^{-1},
\]
so the two stabilizers have the same order.  Moreover,
$a^*\omega=\chi(a)\omega$, and multiplication of a differential by the
nonzero scalar $\chi(a)$ does not change its order.  Pullback by the
isomorphism $a$ gives
\[
 \operatorname{ord}_{\xi'}(\omega)
 =\operatorname{ord}_{\xi}(a^*\omega)
 =\operatorname{ord}_{\xi}(\omega).
\]
Thus $m_\tau$ and $h_\tau$ are independent of $\xi$, and so is
$\sigma_\tau=h_\tau/m_\tau$.
\end{proof}

\begin{lemma}[Signature identity]\label{lem:signature-identity}
For a deformation datum over a curve $X$ of genus $g_X$,
\[
 \sum_{\tau\ \mathrm{critical}}(\sigma_\tau-1)=2g_X-2.
\]
\end{lemma}

\begin{proof}
Put $n=|H|$.  The fiber of $Z\to X$ above $\tau$ has $n/m_\tau$ points.
At each of these points the order of $\omega$ is $h_\tau-1$.  The degree of
the divisor of a meromorphic differential therefore gives
\[
 2g_Z-2=\sum_\tau\frac{n}{m_\tau}(h_\tau-1).
\]
Because $Z\to X$ is tame, Riemann--Hurwitz gives
\[
 2g_Z-2=n(2g_X-2)
 +\sum_\tau\frac{n}{m_\tau}(m_\tau-1).
\]
Subtract the second ramification sum from the first equality and divide by
$n$.  The summand at $\tau$ becomes
\[
 \frac{h_\tau-1}{m_\tau}-\frac{m_\tau-1}{m_\tau}
 =\frac{h_\tau}{m_\tau}-1=\sigma_\tau-1.
\]
Noncritical points have $(m_\tau,h_\tau)=(1,1)$ and contribute zero, which
proves the formula.
\end{proof}

\begin{definition}[Special deformation datum]\label{def:special-datum}
A logarithmic deformation datum $(Z,\omega)$ over $\mathbf P^1_k$ is
special if every critical invariant satisfies $\sigma_\tau<2$ and
$\sigma_\tau\ne1$, and exactly three critical points satisfy
$\sigma_\tau<1$.  If $B$ indexes the critical points, put
\[
 \begin{aligned}
 B_{\rm wild}&=\{j\in B:\sigma_j=0\},\\
 B_{\rm prim}&=\{j\in B:0<\sigma_j<1\},\\
 B_{\rm new}&=\{j\in B:1<\sigma_j<2\},\\
 B_0&=B_{\rm wild}\cup B_{\rm prim}.
 \end{aligned}
\]
Thus $|B_0|=3$.  It is normalized if the three points indexed by $B_0$ are
$0,1,\infty$.  The three kinds of indices are called wild, primitive, and
new, respectively.
\end{definition}

\begin{definition}[Pointed tail cover]\label{def:tail}
Let $G$ be finite.  A $G$-cover means a finite generically separable map
$\bar f:\bar Y\to\bar X$ with a $G$-action on $\bar Y$ and quotient
$\bar Y/G=\bar X$; the source is allowed to be disconnected.  A $G$-tail
cover is a $G$-cover with $\bar X\simeq\mathbf P^1_k$ which is wildly
ramified above a distinguished point
$\bar\infty$ with inertia order $pm$, where $p\nmid m$, and whose restriction to
$\bar X-\{\bar\infty\}$ is tame and branched at at most one point.  A
pointing consists of a point $\eta$ above $\bar\infty$ and a complement
$H\leq I_\eta$ such that
$I_\eta=I_{\eta,1}\rtimes H$, where $I_{\eta,1}\simeq C_p$.  Write
$m=|H|$, let $h$ be the unique lower
jump of $I_{\eta,1}$, and put $\sigma=h/m$.  The tail is special if
$\sigma<2$, $\sigma\ne1$, and its affine restriction is \'etale whenever
$\sigma>1$.  It is primitive if $\sigma<1$ and new if $\sigma>1$.
The ordered pair $(m,h)$ is called the integer type of the pointed tail and
is denoted by $\operatorname{type}(f)$.
\end{definition}

Definitions \ref{def:deformation-datum}--\ref{def:tail} are
\cite[Definitions~1.5, 2.7, 2.9, and 2.10]{WewersThree}.

\begin{definition}[Special $G$-deformation datum]
\label{def:special-G-datum}
Fix a finite group $G$, a subgroup
\[
 G_0=P\rtimes H_0\leq G,\qquad P=\langle\alpha\rangle\simeq C_p,
\]
and the actual conjugation character
\[
 \chi_0:H_0\longrightarrow\mathbf F_p^\times,
 \qquad h\alpha h^{-1}=\alpha^{\chi_0(h)}.
\]
A special $G$-deformation datum consists of the following objects.
\begin{enumerate}
\item A normalized special deformation datum $(Z_0,\omega_0)$ of type
$(H_0,\chi_0)$, with critical points $(\tau_j)_{j\in B}$ and integer types
$(m_j,h_j)$.
\item For every $j\in B-B_{\rm wild}$, a point $\xi_j\in Z_0$ above
$\tau_j$.
\item For every $j\in B-B_{\rm wild}$, a subgroup $G_j\leq G$ and a
connected pointed $G_j$-tail
$(f_j:Y_j\to\mathbf P^1,\eta_j,H_j)$.
\end{enumerate}
These objects are required to satisfy
\[
 G=\langle G_0,G_j:j\in B-B_{\rm wild}\rangle,
\]
and, for every $j\in B-B_{\rm wild}$,
\[
 \operatorname{type}(f_j)=(m_j,h_j),\qquad
 I_{\eta_j}=P\rtimes H_j,\qquad
 H_j=\operatorname{Stab}_{H_0}(\xi_j)
     =H_0\cap I_{\eta_j}.
\]
\end{definition}

This is \cite[Definition~2.13]{WewersThree}.

\begin{lemma}[Tame inertia under lifting]\label{lem:tame-profile}
Let $R$ be a complete discrete valuation ring with algebraically closed
residue field of characteristic $p$, and let $\mathcal D$ be a formal disk
over $R$ with a section $x_R$.  Let
$\mathcal E\to\mathcal D$ be the unique tame lift of a finite
$G$-equivariant cover of the special fiber which is \'etale away from the
specialization of $x_R$.  If $T\leq G$ is the inertia group at that point,
then the generic inertia group along $x_R$ is conjugate to $T$.  Its
generators differ from generators of $T$ by powers prime to $|T|$.
Consequently, for every finite $G$-set $\Omega$, the orbit partition of a
local inertia generator on $\Omega$ is unchanged by the lift.  If the
special-fiber cover is \'etale, its unique tame lift is \'etale.
\end{lemma}

\begin{proof}
After strict henselization at the section, a connected tame local cover of
ramification index $e$ has the form
\[
 R^{\mathrm{sh}}[[t]]\longrightarrow R^{\mathrm{sh}}[[s]],
 \qquad t=u s^e,
\]
where $u$ is a unit and $p\nmid e$.  The tame fundamental group is
procyclic.  Because $e$ is invertible and the ring is strictly henselian,
Hensel's lemma gives an $e$th root of $u$; after replacing $s$ by that root
times $s$, the equation is $t=s^e$.  Reduction induces a bijection between
the $e$th roots of unity in $R^{\mathrm{sh}}$ and those in its residue
field.  Thus the same element
\[
 s\longmapsto\zeta_e s
\]
generates the special and generic inertia groups.  Changing either chosen
generator replaces it by its $a$th power for some
$a\in(\mathbf Z/e\mathbf Z)^\times$.  On a finite $G$-set, a generator and
such a power generate the same cyclic subgroup; hence they have the same
orbits, and on each orbit both act as a cycle whose length is the orbit
cardinality.  Their cycle partitions are therefore equal.  When $e=1$, the
local equation is $t=s$ and the extension is \'etale.  Applying this argument
to every connected component proves the statement.
\end{proof}

\begin{theorem}[Lifting interface]\label{prop:interface}
Let $p$ be odd and let $k$ be algebraically closed of characteristic $p$.
Let $G\leq S_p$ satisfy $p\Vert |G|$.  Fix a regular subgroup
$P=C_p\leq G$ and an actual complement $H_0\leq N_{S_p}(P)$ such that
\[
 G_0=P\rtimes H_0\leq G.
\]
Let $\chi_0:H_0\hookrightarrow\operatorname{Aut}(P)
=\mathbf F_p^\times$ be the character induced by conjugation.
Let a special $G$-deformation datum in the sense of Definition
\ref{def:special-G-datum} be given.  Assume that its central deformation
datum is logarithmic, has no wild critical point, and hence has exactly
three primitive critical points.  Then it defines a connected special
stable $G$-map, and that map is
the stable reduction of a characteristic-zero three-point $G$-cover.
Moreover, on every primitive tail the marked tame restriction away from the
wild attachment point is preserved by the lift; in particular, its finite
tame inertia action on every $G$-set is preserved.  A new tail creates no
additional horizontal branch point.
\end{theorem}

\begin{proof}
Let $X_0=\mathbf P^1_k$.  Attach a copy $X_j=\mathbf P^1_k$ to $X_0$ by
identifying the distinguished point $\infty_j\in X_j$ with
$\tau_j\in X_0$.  Denote the resulting nodal curve by $\bar X$.  The
logarithmic datum determines the central inseparable cover $Y_0\to X_0$
used in \cite[Definitions~2.14--2.15]{WewersThree}.  Form the induced
$G$-curves
\[
 \operatorname{Ind}_{G_0}^{G}(Y_0)
 \quad\text{and}\quad
 \operatorname{Ind}_{G_j}^{G}(Y_j)
\]
for all $j$.  At the node above $\tau_j=\infty_j$, the equalities
$I_{\eta_j}=P\rtimes H_j$ and $H_j=H_0\cap I_{\eta_j}$ identify the two
induced $G$-orbits with the same stabilizer.  They therefore give a
$G$-equivariant gluing.  The connected components of the resulting source
are permuted by the subgroup generated by $G_0$ and all $G_j$.  This subgroup
is $G$, so the source is connected.  Thus Definitions~2.14--2.15 of
\cite{WewersThree} give a connected special stable $G$-map
$\bar Y\to\bar X$.

Corollary 4.6 of \cite{WewersThree} lifts every such special map to a
three-point $G$-cover.  More precisely, in the proof of Theorem 4.5 in
\cite[Section 4.2.3]{WewersThree}, Wewers deletes the wild attachment point
from a tail, passes to the resulting induced, possibly disconnected, affine
tame cover, and uses its unique tame lift over the corresponding formal disk;
the lift is unramified away from the horizontal section.  Lemma
\ref{lem:tame-profile} shows that, at a primitive tail, the finite tame
inertia subgroup and its orbit partition on every $G$-set are preserved.  At
a new tail the affine special-fiber cover is \'etale, so its unique tame lift
is \'etale and gives no horizontal ramification.
Wewers marks the base by $B_0=B_{\rm prim}\cup B_{\rm wild}$.  In the
present hypotheses $B_{\rm wild}=\varnothing$ and $|B_{\rm prim}|=3$, so the
generic cover has exactly those three branch sections.
\end{proof}

\section{Primitive tails with prescribed profile}

\begin{theorem}[Henrio's moment criterion]\label{thm:henrio-moments}
Let $r$ be an integer with $2\leq r\leq p$.  Let
$e_1,\ldots,e_r\in\mathbf F_p^\times$ satisfy
\[
 \sum_j e_j=0.
\]
Suppose that no nonempty proper subcollection of the
$e_j$ has sum zero.  Then there are pairwise distinct
$a_1,\ldots,a_r\in\overline{\mathbf F}_p$ such that
\begin{equation}\label{eq:henrio-moments}
 \sum_{j=1}^{r}e_j a_j^q=0\qquad(1\leq q\leq r-2).
\end{equation}
\end{theorem}

\begin{proof}
Call a set partition of $\{1,\ldots,r\}$ adapted if the sum of the $e_j$
over each of its blocks is zero in $\mathbf F_p$.  The one-block set
partition is adapted because $\sum_j e_j=0$.  If an adapted partition had
at least two blocks, any one of its blocks would be a nonempty proper
subcollection with sum zero, contrary to the hypothesis.  Hence the
one-block partition is the unique adapted partition and, in particular, the
unique maximal adapted partition in the terminology of Henrio.

Set $m_H=r-1$.  Since $r\leq p$, one has $m_H<p$ and therefore
\[
 1\leq m_H\leq p-1,\qquad \gcd(m_H,p)=1,\qquad
 \left\lfloor\frac{m_H}{p}\right\rfloor+1=1,
\]
which is the number of blocks of the unique maximal adapted partition.
Thus every hypothesis in Henrio's terminology
\cite[Definition~3.14 and Proposition~3.16(b)]{Henrio} is satisfied.
That proposition gives pairwise distinct coordinates satisfying the moment
equations for $1\leq q\leq m_H-1=r-2$ with $p\nmid q$.  The final condition
is automatic because $1\leq q\leq r-2<p$.  When $r=2$, this range of
integers $q$ is empty; the conclusion of the proposition is then precisely
the existence of two distinct coordinates.
\end{proof}

\begin{lemma}[Moments and a derivative identity]
\label{lem:moments-derivative}
Let $k$ be a field, let $a_1,\ldots,a_r\in k$ be pairwise distinct, and let
$e_1,\ldots,e_r$ be positive integers whose images in $k$ are nonzero.  Put
\[
 A(y)=\prod_{j=1}^r(y-a_j),\qquad
 P(y)=\prod_{j=1}^r(y-a_j)^{e_j}.
\]
If
\[
 \sum_{j=1}^r e_ja_j^q=0\qquad(0\leq q\leq r-2),
\]
then there is a scalar $c\in k^\times$ such that
\[
 P'(y)=c\prod_{j=1}^r(y-a_j)^{e_j-1}.
\]
\end{lemma}

\begin{proof}
The product rule gives
\[
 P'(y)=
 \left(\prod_{j=1}^r(y-a_j)^{e_j-1}\right)S(y),
 \qquad
 S(y)=\sum_{j=1}^r e_j\prod_{\ell\ne j}(y-a_\ell).
\]
In the Laurent-series field $k((y^{-1}))$,
\[
 \frac{S(y)}{A(y)}
 =\sum_{j=1}^r\frac{e_j}{y-a_j}
 =\sum_{q\geq0}
   \left(\sum_{j=1}^r e_ja_j^q\right)y^{-q-1}.
\]
The assumed moments show that the right side belongs to
$y^{-r}k[[y^{-1}]]$.  Since $A$ is monic of degree $r$, the product
$S=A(S/A)$ belongs to $k[[y^{-1}]]$.  It is also a polynomial in $y$, so
it is constant.  For every $i$,
\[
 S(a_i)=e_i\prod_{\ell\ne i}(a_i-a_\ell)\ne0.
\]
Thus the constant $c=S$ is nonzero, and the product-rule identity gives the
asserted formula.
\end{proof}

\begin{lemma}[Henrio polynomial]\label{lem:primitive}
Let $\Lambda=(e_1,\ldots,e_r)\vdash p$ with $2\leq r\leq p-1$.
Over $k=\Fp$ there are distinct $a_1,\ldots,a_r$ and a polynomial
\[
 P_\Lambda(y)=\prod_{j=1}^r(y-a_j)^{e_j}
\]
such that
\[
 P_\Lambda'(y)=c\prod_{j=1}^r(y-a_j)^{e_j-1},\qquad c\neq0.
 \label{eq:primitive-derivative}
\]
Consequently $P_\Lambda:\PP^1\to\PP^1$ is a separable degree-$p$
cover with exactly two branch values: profile $\Lambda$ over $0$, and one
totally wildly ramified point over $\infty$.
\end{lemma}

\begin{proof}
Regard the positive integers $e_j<p$ as elements of $\mathbf F_p^\times$.
Their total is zero, while no nonempty proper subcollection has zero sum.
Theorem \ref{thm:henrio-moments} gives distinct $a_j$ satisfying
\eqref{eq:henrio-moments}.
The zeroth moment also vanishes because $\sum_j e_j=p=0$ in $k$.  Lemma
\ref{lem:moments-derivative} now gives
\eqref{eq:primitive-derivative}.  The degree of $P_\Lambda$ is
$\sum_j e_j=p$, and its derivative is nonzero, so the induced extension of
function fields is separable.  At $y=a_j$ one has
\[
 P_\Lambda(y)=(y-a_j)^{e_j}u_j(y),\qquad u_j(a_j)\ne0.
\]
Thus the local degree is $e_j$ and the image is zero.  If a finite point
$b$ is distinct from every $a_j$, then
$P_\Lambda'(b)\ne0$, so the morphism is unramified at $b$.  Finally, with
the parameter $s=1/y$ at infinity,
$\operatorname{ord}_\infty(P_\Lambda)=-p$.  Therefore infinity is the
unique point above infinity and has ramification index $p$.  Since
$r<p$, at least one $e_j$ exceeds one, so zero is a genuine branch value.
The two branch values are exactly $0$ and $\infty$, with the stated
profiles.
\end{proof}

\begin{lemma}[Local conductor and character]\label{lem:local}
Let $K$ be a complete discretely valued field with algebraically closed
residue field $k$ of characteristic $p$, and let $E/K$ be a totally ramified
separable extension of degree $p$.  Let $L/K$ be its Galois closure.  Suppose
that the inertia group of $L/K$, acting on the $p$ $K$-embeddings of $E$ in
$L$, is
\[
 I=P\rtimes H\leq S_p,\qquad P=C_p,\quad H=C_m,
\]
and that $P$ has unique lower jump $h$.  Put $M=L^P$.  If $\pi_M$ is a
uniformizer of $M$, define the tame parameter character by
\[
 a(\pi_M)\equiv\psi(a)\pi_M\pmod{\pi_M^2}\qquad(a\in H).
\]
Let $d_{\rm loc}=v_E(\mathfrak D_{E/K})$ be the exponent of the different.
If $\chi:H\to\operatorname{Aut}(P)$ is the conjugation character, then
\[
 d_{\rm loc}=(p-1)\left(1+\frac hm\right),\qquad
 \chi=\psi^{\varepsilon h}\quad\text{for }\varepsilon\in\{1,-1\}.
 \label{eq:local-type}
\]
If the natural action is faithful, then $\chi$ is faithful and
$\gcd(h,m)=1$.
\end{lemma}

\begin{proof}
Fix a prime $\ell\ne p$.  Let $A=\mathbf Q_\ell^{\,p}$ be the permutation
representation on the $p$ embeddings and put
$V=A/\mathbf Q_\ell\mathbf 1$.  If a subgroup is transitive on the $p$
letters, its invariant subspace in $A$ consists precisely of the constant
vectors.  Both $I$ and its regular subgroup $P$ are transitive, so
$V^I=V^P=0$.  The lower ramification groups are
$I_0=I$, $I_s=P$ for $1\leq s\leq h$, and trivial thereafter.  The Artin
conductor formula
\[
 a(V)=\operatorname{codim}V^I+
 \sum_{s\geq1}\frac{|I_s|}{|I_0|}\operatorname{codim}V^{I_s}
\]
therefore gives
\[
 a(V)=(p-1)+
 \sum_{s=1}^{h}\frac{|P|}{|I|}(p-1)
 =(p-1)\left(1+\frac hm\right).
\]
The representation $A$ is the induced representation attached to the
degree-$p$ subextension.  If $\mathfrak d_{E/K}$ denotes the discriminant
ideal, the conductor--discriminant formula gives
\[
 a(A)=v_K(\mathfrak d_{E/K}).
\]
The different and discriminant ideals satisfy
\[
 \mathfrak d_{E/K}=N_{E/K}(\mathfrak D_{E/K}).
\]
Since $E/K$ is totally
ramified, its residue degree $f(E/K)$ equals one; hence
\[
 v_K(\mathfrak d_{E/K})
 =f(E/K)v_E(\mathfrak D_{E/K})=d_{\rm loc}.
\]
Since
$A=\mathbf Q_\ell\mathbf 1\oplus V$ and the trivial summand has conductor
zero, $a(A)=a(V)$.  This proves the different formula.
Fix the convention for the conjugation character and the tame parameter
character used in \cite[Lemma 2.12]{WewersThree}.  Equivariance of its local
Artin--Schreier normal form gives the second equality in
\eqref{eq:local-type}; the sign $\varepsilon$ records the character
convention.  The character $\psi$ has
order $m$ because $H$ is the tame inertia quotient.  Moreover, an element in
$\ker\chi$ centralizes the regular $p$-cycle;
since $C_{S_p}(P)=P$ and $H\cap P=1$, faithfulness of the permutation action
forces $\ker\chi=1$.  Hence $\chi$ has order $m$.  The order of
$\psi^{\pm h}$ is $m/\gcd(m,h)$.  The equality
$\chi=\psi^{\pm h}$ therefore implies $m/\gcd(m,h)=m$, or
$\gcd(h,m)=1$.
\end{proof}

\begin{lemma}[Primitive-tail verification]\label{lem:primitive-tail}
The connected Galois closure of $P_\Lambda$ is a connected pointed special
primitive tail.  In its faithful natural action $G_\Lambda\leq S_p$, the
inertia at infinity is
\[
 I_\Lambda=P\rtimes H_\Lambda,\qquad
 P=C_p,\quad H_\Lambda=C_{m_\Lambda},
 \label{eq:primitive-inertia}
\]
and, if $h_\Lambda$ is the positive lower jump,
\[
 \frac{h_\Lambda}{m_\Lambda}
 =\frac{r-1}{p-1},\qquad
 m_\Lambda=\frac{p-1}{\gcd(p-1,r-1)},\quad
 h_\Lambda=\frac{r-1}{\gcd(p-1,r-1)}.
 \label{eq:primitive-type}
\]
Its unique finite tame inertia has complete cycle partition $\Lambda$ in
the same natural $p$-letter action.
\end{lemma}

\begin{proof}
Take the normal closure as a field and denote its smooth projective Galois
cover by $f_\Lambda:Y_\Lambda\to\mathbf P^1$.  Its source is connected and
its Galois group acts faithfully on the $p$ embeddings because an
automorphism acting trivially on every embedding fixes the normal closure.
The cover $P_\Lambda$ is étale away from $0,\infty$; Lemma
\ref{lem:normal-closure-etale}, applied to that open curve, shows that its
normal closure does not add branch points there.  At $0$ all local indices
$e_j$ are prime to $p$, and a compositum of tame local extensions is tame.
Choose a point $\eta$ of the connected normal closure above infinity and
write $I_\eta$ for its inertia group.  In the degree-$p$ quotient there is a
unique point above infinity.  Lemma \ref{lem:quotient-profile} therefore
shows that the decomposition group at $\eta$ has a single orbit on the $p$
embeddings, so it is transitive.  Since the
constant field $k$ is algebraically closed, all residue extensions are
trivial; hence decomposition equals inertia and $I_\eta$ is transitive.
The orbit of a letter under $I_\eta$ has cardinality $p$, so $p$ divides
$|I_\eta|$.  Because $I_\eta\leq S_p$ and $v_p(p!)=1$, a Sylow
$p$-subgroup has order $p$; it is generated by a $p$-cycle and is therefore
a regular $P=C_p$.  Inertia theory over an algebraically closed residue
field makes the tame quotient $I_\eta/P$ cyclic.  Since its order is prime
to $p$, Schur--Zassenhaus supplies a complement
$H_\Lambda=C_{m_\Lambda}$ with
$I_\eta=P\rtimes H_\Lambda$.  Thus
$(f_\Lambda,\eta,H_\Lambda)$ is a connected pointed tail cover.

At the point $a_j$ the tame different exponent is $e_j-1$.  Hence the total
finite different is
$\sum_j(e_j-1)=p-r$.  Riemann--Hurwitz on
$P_\Lambda:\mathbf P^1\to\mathbf P^1$ says that the total different has
degree $2p-2$, so the different exponent at infinity is
$2p-2-(p-r)=p+r-2$.  Lemma \ref{lem:local} gives
\[
 p+r-2=(p-1)\left(1+\frac{h_\Lambda}{m_\Lambda}\right),
\]
and subtraction and division by $p-1$ give
$h_\Lambda/m_\Lambda=(r-1)/(p-1)$.  The character argument in that lemma
shows
$\gcd(h_\Lambda,m_\Lambda)=1$.  Hence the reduced numerator and denominator
of this rational number are the actual local integers, namely the two
integers in \eqref{eq:primitive-type}.  Since
$2\leq r\leq p-1$, one has
$0<h_\Lambda/m_\Lambda<1$.  Together with the single finite tame branch
point (which is genuine because $r<p$, so some $e_j>1$), this is exactly a
special primitive tail under
\cite[Definitions 2.9--2.10]{WewersThree}.  Finally, apply Lemma
\ref{lem:quotient-profile} to the degree-$p$ quotient at zero.  Its points
above zero have local degrees $e_1,\ldots,e_r$, so a generator of the tame
inertia has exactly these cycle lengths in the natural $p$-letter action.
This proves the complete profile assertion.
\end{proof}

\section{The positive-genus new tail}

\begin{lemma}[Hyperelliptic new tail]\label{lem:newtail}
Assume $g>0$.  There is a connected pointed special new tail whose natural
degree-$p$ quotient has genus $g$ and is étale over $\mathbf A^1$.  If its
inertia at infinity is
$P\rtimes H_g$ and $h_g$ is the positive lower jump, then, with
\[
 d_g=\gcd(p-1,2g),
\]
its actual local data are
\[
 |H_g|=m_g=\frac{p-1}{d_g},\qquad
 h_g=\frac{p-1+2g}{d_g},\qquad
 \sigma_g=\frac{h_g}{m_g}=1+\frac{2g}{p-1}.
 \label{eq:new-type}
\]
The character $H_g\to\operatorname{Aut}(P)$ is the faithful multiplier
character in the natural embedding in $S_p$.
\end{lemma}

\begin{proof}
The residual bound \eqref{eq:genus-bound} gives
$p\geq7$, $0<2g\leq p-5$, and
\[
 n=\frac{p-1}{2}-g\geq2,\qquad
 \rho_g=(2^n,1^{2g+1})\vdash p.
 \label{eq:aux-partition}
\]
The displayed family is a partition because
$2n+(2g+1)=p$.  Its length is
$n+2g+1=(p+1)/2+g$, so its length minus one is
$(p-1)/2+g\leq p-3$.  A nonempty proper subcollection of its parts has a
positive integral sum strictly smaller than $p$, and therefore does not sum
to zero in $\mathbf F_p$.  Theorem \ref{thm:henrio-moments} applies and
supplies pairwise distinct $b_j,c_k$ such that
\[
 H(x)=\prod_{j=1}^n(x-b_j),\quad
 f(x)=\prod_{k=1}^{2g+1}(x-c_k),\quad
 R=H^2f,\quad R'=cH,\quad c\neq0.
 \label{eq:new-derivative}
\]
Indeed, the zeroth moment is the sum of the parts of $\rho_g$, hence is
$p=0$ in $k$.  Lemma \ref{lem:moments-derivative} applied to the exponents
of $\rho_g$ gives the derivative identity: a root $b_j$ has exponent $2$
and contributes one factor to the derivative, while a root $c_k$ has
exponent $1$ and contributes none.

The roots $c_k$ are distinct, so $f$ is squarefree.  Because $p$ is odd, the
affine curve $y^2=f(x)$ is smooth.  Its smooth projective model
\[
 C_g:y^2=f(x)
\]
is connected because $f$ is not a square.  The double cover
$x:C_g\to\mathbf P^1_x$ is ramified at the $2g+1$ roots of $f$ and at
infinity.  It therefore has $2g+2$ simple ramification points, and
Riemann--Hurwitz gives
\[
 2g(C_g)-2=2(-2)+(2g+2)=2g-2.
\]
Thus $g(C_g)=g$.  Odd degree gives a unique point $Q$ over infinity.  Put
$u=H(x)y$.  Cancelling $H$ from $R'=cH$ gives
$2H'f+Hf'=c$, and hence
\[
 du=H'y\,dx+\frac{Hf'}{2y}\,dx
   =\frac{c}{2}\frac{dx}{y}.
 \label{eq:new-differential}
\]
At an affine point with $y\ne0$, the function $x$ is a local parameter and
$dx/y$ is nonzero.  At a point $(c_k,0)$, the squarefreeness of $f$ gives
$f'(c_k)\ne0$; the function $y$ is a local parameter and differentiating
$y^2=f(x)$ gives
\[
 \frac{dx}{y}=\frac{2\,dy}{f'(x)},
\]
which is again regular and nonzero.  At the unique point $Q$ over infinity,
the degree-two map $x$ and the equation $y^2=f(x)$ give
\[
 \operatorname{ord}_Q(x)=-2,\quad
 \operatorname{ord}_Q(y)=-(2g+1).
\]
Since $p$ is odd, the coefficient $2$ is nonzero in $k$.  Differentiating
the leading term of a function with a pole of order $2$ therefore gives
\[
 \operatorname{ord}_Q(dx)=-3.
\]
Hence
\[
 \operatorname{ord}_Q(dx/y)=-3+(2g+1)=2g-2.
\]
There are no other zeros or poles, so
$\operatorname{div}(dx/y)=(2g-2)Q$.  Equation
\eqref{eq:new-differential} and $c\ne0$ show that $du$ is nonzero and has no
zero at an affine point.  Therefore the morphism defined by $u$ is separable
and unramified over the affine line.  Moreover
\[
 \operatorname{ord}_Q(u)=-(2n+2g+1)=-p.
\]
Thus $u$ has the unique pole divisor $pQ$.  A nonconstant rational function
on a smooth projective curve defines a finite morphism whose degree is the
degree of its pole divisor.  Hence $u$ defines a finite separable morphism of
degree $p$,
\[
 u:C_g\longrightarrow\mathbf P^1_u
 \label{eq:new-map}
\]
whose affine restriction is finite and unramified.  A finite separable
unramified morphism between smooth curves is \'etale.  The pole divisor
$pQ$ also shows that $Q$ is its sole point over infinity, with ramification
index $p$.

Take the connected function-field normal closure and call its natural group
$G_g\leq S_p$.  Lemma \ref{lem:normal-closure-etale} shows that normal
closure does not introduce ramification over $\mathbf A^1$.  Choose a point
$\eta$ over infinity.  By Lemma \ref{lem:quotient-profile}, the unique point
of the degree-$p$ quotient over infinity makes its decomposition group
transitive on the $p$ embeddings.  Because $k$ is algebraically closed,
decomposition equals inertia.  Its order is divisible by $p$; since
$v_p(p!)=1$, its Sylow $p$-subgroup has order $p$ and is generated by a
$p$-cycle, hence is a regular $P=C_p$.  The tame inertia quotient is cyclic,
and Schur--Zassenhaus gives a complement $H_g$.  This makes the normal
closure a connected pointed
$G_g$-tail in the precise sense of
\cite[Definition 2.9]{WewersThree}.

The map \eqref{eq:new-map} is \'etale away from infinity.  Its sole different
exponent is therefore determined by Riemann--Hurwitz:
\[
 2g-2=p(-2)+d_\infty,
\]
so $d_\infty=2g+2p-2$.  Lemma \ref{lem:local} gives
\[
 2g+2p-2=(p-1)\left(1+\frac{h_g}{m_g}\right).
\]
Solving this equality gives
$h_g/m_g=1+2g/(p-1)$.  Its faithful-character conclusion shows
$\gcd(h_g,m_g)=1$, so the reduced numerator and denominator are exactly the
integers in \eqref{eq:new-type}.  Finally \eqref{eq:genus-bound} gives
$1<\sigma_g<2$, while the
affine restriction is étale.
Therefore \cite[Definition 2.10]{WewersThree} calls this a special new tail.
Under the identification of the $p$ letters with $\mathbf F_p$, the regular
subgroup $P$ acts by translations.  Lemma \ref{lem:normalizer} identifies
the conjugation action of $H_g$ on $P$ with its multiplier in
$\mathbf F_p^\times$.  Its kernel would centralize $P$ and meet $P$
trivially, so the same lemma shows that the character is faithful.
\end{proof}

\section{The central logarithmic datum}

\begin{definition}[Cyclic special degeneration datum]
\label{def:cyclic-degeneration}
Let $m>1$ divide $p-1$, let
$\chi:C_m\hookrightarrow\mathbf F_p^\times$ be faithful, and let
$\tau_1,\ldots,\tau_r\in\mathbf P^1(k)$ be pairwise distinct.  Choose
integers
\[
 0<a_j<m,\qquad \sum_{j=1}^r a_j=m,
\]
and assume that
\[
 \gcd(m,a_1,\ldots,a_r)=1.
\]
Let $Z\to\mathbf P^1_k$ be the connected cyclic cover of degree $m$
with these Kummer exponents, the deck action being chosen so that the
Kummer coordinate has character $\chi$.  Put
\[
 c_j=\gcd(m,a_j),\qquad m_j=\frac{m}{c_j},\qquad
 \widetilde a_j=\frac{a_j}{c_j},
\]
and let $P_j$ be the reduced divisor above $\tau_j$.  Choose
$\nu_j\in\{0,1\}$ with $\sum_j\nu_j=r-3$.

A cyclic special degeneration datum is a nonzero element of
\[
 H^0(Z,\Omega^1_{Z/k})_\chi
\]
which, denoted by $\omega$, satisfies
\begin{align}
 \operatorname{div}(\omega)
 &=\sum_{j=1}^r(m_j\nu_j+\widetilde a_j-1)P_j,
 \label{eq:degeneration-divisor}\\
 \mathcal C(\omega)&=\omega,
 \label{eq:degeneration-cartier}
\end{align}
where $\mathcal C$ is the Cartier operator.
\end{definition}

The numerical, divisor, and Cartier conditions are those of
\cite[Definition~3.1 and equations~(29)--(30)]{WewersMeta}; the connectedness
condition needed for a deformation datum has been imposed explicitly.

\begin{lemma}[Kummer fibers and critical types]
\label{lem:kummer-critical-type}
In Definition \ref{def:cyclic-degeneration}, the fiber above $\tau_j$ has
exactly $c_j$ points, each of ramification index $m_j$.  If $Q$ is one of
these points, there are local parameters $t$ at $\tau_j$ and $s$ at $Q$
such that
\[
 t=s^{m_j},\qquad z_j=u_js^{\widetilde a_j},\qquad
 u_j\in k[[s]]^\times,
\]
where $z_j$ differs from a Kummer generator by a power of $t$ and a unit.
The critical type of a special degeneration datum at $\tau_j$ is
\[
 h_j=m_j\nu_j+\widetilde a_j,
 \qquad
 \sigma_j=\frac{h_j}{m_j}=\nu_j+\frac{a_j}{m}.
 \label{eq:kummer-critical-type}
\]
The stabilizer of $Q$ in $C_m$ is the unique subgroup of order $m_j$, and
the character at that stabilizer is the restriction of $\chi$.
\end{lemma}

\begin{proof}
Choose a parameter $t$ at $\tau_j$.  After multiplying the Kummer generator
by an integral power of $t$, its completed local equation is
\[
 z_j^m=t^{a_j}v(t),\qquad v(t)\in k[[t]]^\times.
\]
The integer $m$ is prime to $p$, and $k$ is algebraically closed.  Hensel's
lemma therefore supplies an $m$th root of $v(t)$; division by that root
reduces the equation to $z_j^m=t^{a_j}$.  Write
$m=c_jm_j$ and $a_j=c_j\widetilde a_j$.  Then
\[
 z_j^m-t^{a_j}
 =\prod_{\zeta^{c_j}=1}
   \bigl(z_j^{m_j}-\zeta t^{\widetilde a_j}\bigr).
\]
There are $c_j$ factors.  Since
$\gcd(m_j,\widetilde a_j)=1$, the normalization of each factor is one
formal disk, and after multiplying $z_j$ by a constant it has a parameter
$s$ satisfying $t=s^{m_j}$ and $z_j=s^{\widetilde a_j}$.  This proves the
fiber and ramification assertions.

The coefficient of $Q$ in \eqref{eq:degeneration-divisor} is
$m_j\nu_j+\widetilde a_j-1$.  Adding one and dividing by $m_j$ gives
\eqref{eq:kummer-critical-type}.  The stabilizer order equals the
ramification index because the residue field is algebraically closed.
The group $C_m$ is cyclic, so it has a unique subgroup of that order.  The
character statement follows by restricting the character of the deck
action.
\end{proof}

\begin{theorem}[Cartier identities and logarithmic criterion]
\label{thm:cartier-criterion}
Let $K_0$ be a function field of one variable over an algebraically closed
field of characteristic $p$.  For $f\in K_0$ and a rational differential $\eta$,
the Cartier operator satisfies
\[
 \mathcal C(f^p\eta)=f\,\mathcal C(\eta).
\]
Moreover,
\[
 \mathcal C(\eta)=\eta
 \quad\Longleftrightarrow\quad
 \eta=\frac{du}{u}\ \text{for some }u\in K_0^\times.
\]
The Cartier operator commutes with pullback by an isomorphism of smooth
curves over the ground field.
\end{theorem}

\begin{proof}
The semilinearity formula and the identities
$\mathcal C(dv)=0$ and $\mathcal C(dv/v)=dv/v$ are recorded in
\cite[equations~(40)--(42)]{WewersMeta}.  The logarithmic Cartier criterion
is the statement immediately preceding
\cite[Definition~3.1]{WewersMeta}; equivalently, it is the equality
\[
 \ker(\mathcal C-1: \Omega^1_{K_0}\longrightarrow\Omega^1_{K_0})
   =d\log(K_0^\times).
\]
Finally, if $\varphi:K_0\xrightarrow{\sim}K_1$ is induced by an isomorphism
of smooth curves, applying $\varphi^*$ to the three displayed Cartier
identities characterizes the Cartier operator on $K_1$.  Hence
$\varphi^*\mathcal C=\mathcal C\varphi^*$.
\end{proof}

\begin{lemma}[From degeneration data to deformation data]
\label{lem:degeneration-to-deformation}
A cyclic special degeneration datum is a logarithmic special deformation
datum.  Its primitive critical points are precisely the three points with
$\nu_j=0$, and its new critical points are precisely the points with
$\nu_j=1$.
\end{lemma}

\begin{proof}
Theorem \ref{thm:cartier-criterion} and
\eqref{eq:degeneration-cartier} make the datum logarithmic.  Lemma
\ref{lem:kummer-critical-type} gives
$\sigma_j=\nu_j+a_j/m$.  Since $0<a_j<m$, this invariant lies in $(0,1)$
when $\nu_j=0$ and in $(1,2)$ when $\nu_j=1$.  Because every $\nu_j$ is
zero or one and their sum is $r-3$, exactly three are zero.  Away from the
$\tau_j$, the cover is unramified and \eqref{eq:degeneration-divisor} gives
order zero for $\omega$, so the type is $(1,1)$ and the point is not
critical.  These facts verify Definition \ref{def:special-datum}.
\end{proof}

\begin{theorem}[Wewers's cyclic-cover calculations]
\label{thm:wewers-cyclic}
Let $m>1$ divide $p-1$, let
$\chi:C_m\hookrightarrow\mathbf F_p^\times$ be faithful, and let
$0<a_j<m$ satisfy $\sum_j a_j=m$ and
$\gcd(m,a_1,\ldots,a_r)=1$.
\begin{enumerate}
\item For every ordered tuple of pairwise distinct points
$\tau_1,\ldots,\tau_r$, let $Z$ be the associated connected cyclic cover.
Then
\[
 \dim_kH^0(Z,\Omega^1_{Z/k})_\chi=r-2,
\]
and the Cartier operator is bijective on this eigenspace.
\item Suppose $r=4$.  Fix $\nu_j\in\{0,1\}$ with
$\sum_j\nu_j=1$, reorder the indices so that $\nu_1=1$, and put
$\alpha=(p-1)/m$.  Let the ordered branch points and the differential vary.
Modulo $C_m$-equivariant isomorphism and multiplication of the differential
by an element of $\mathbf F_p^\times$, there are exactly
$\alpha a_1-1$ cyclic special degeneration data with these fixed integers
and this fixed character.
\end{enumerate}
\end{theorem}

\begin{proof}
For Part (1), choose a projective coordinate change carrying none of the
$\tau_j$ to infinity.  In that affine coordinate, the assertion is
Lemma~1.4(ii),(iv) of Wewers \cite{WewersMeta}.  Pullback by the coordinate change
identifies the connected cyclic covers and their $\chi$-eigenspaces, and
Theorem \ref{thm:cartier-criterion} shows that it also identifies their
Cartier operators.  This proves the stated coordinate-free form.

The count in Part (2), with the faithful character fixed, is
\cite[Proposition~3.8]{WewersMeta}.  The equivalence relation is the one in
the opening paragraph of \cite[Section~3.5]{WewersMeta}.  After fixing
three ordered branch points at $0,1,\infty$,
the fourth is a variable cross-ratio; it is not fixed in the count.
\end{proof}

\begin{lemma}[Cartier rescaling]\label{lem:cartier-rescaling}
Let $L$ be a one-dimensional $k$-space of differentials preserved
bijectively by Cartier.  For every nonzero $\omega\in L$, there is
$\lambda\in k^\times$ such that
$\mathcal C(\lambda\omega)=\lambda\omega$.
\end{lemma}

\begin{proof}
Write $\mathcal C(\omega)=c\omega$ with $c\in k^\times$.  Choose
$\lambda\in k^\times$ with $\lambda^{p-1}=c^p$, and choose
$\rho\in k^\times$ with $\rho^p=\lambda$.  The $p^{-1}$-semilinearity of
Cartier gives
\[
 \mathcal C(\lambda\omega)
 =\mathcal C(\rho^p\omega)=\rho\mathcal C(\omega)=\rho c\omega.
\]
The scalar equation and $\lambda=\rho^p$ imply
$(\rho^{p-1})^p=c^p$.  Frobenius is injective on the field $k$, so
$\rho^{p-1}=c$ and $\rho c=\rho^p=\lambda$.  This proves the assertion.
\end{proof}

Put
\[
 \sigma_i=\frac{r_i-1}{p-1}.
\]
Equation \eqref{eq:three-RH} becomes
\[
 \sigma_1+\sigma_2+\sigma_3=1-\frac{2g}{p-1}.
 \label{eq:signature-sum}
\]

Put $A_i=r_i-1$ for $1\leq i\leq3$.  If $g=0$, set
$J=\{1,2,3\}$.  If $g>0$, set $A_4=2g$ and
$J=\{1,2,3,4\}$.  In both cases define
\begin{equation}\label{eq:central-arithmetic}
 D=\gcd\bigl(p-1,(A_j)_{j\in J}\bigr),\qquad
 M=\frac{p-1}{D},\qquad a_j=\frac{A_j}{D}.
\end{equation}
Define $\nu_i=0$ for $1\leq i\leq3$, and define $\nu_4=1$ when $g>0$.

\begin{lemma}[Arithmetic of the central exponents]
\label{lem:central-arithmetic}
The preceding integers satisfy
\begin{gather}
 M>1,\qquad 0<a_j<M\quad(j\in J),\\
 \sum_{j\in J}a_j=M,\qquad
 \gcd\bigl(M,(a_j)_{j\in J}\bigr)=1.
 \label{eq:central-bounds}
\end{gather}
For $j\in J$, put
\[
 c_j=\gcd(M,a_j),\qquad m_j=\frac{M}{c_j},\qquad
 \widetilde a_j=\frac{a_j}{c_j},\qquad
 h_j=m_j\nu_j+\widetilde a_j.
\]
Then
\begin{gather}
 \gcd(m_j,\widetilde a_j)=1,\qquad
 \frac{h_j}{m_j}=\nu_j+\frac{A_j}{p-1},
 \label{eq:central-fractions}\\
 \operatorname{lcm}_{j\in J}m_j=M.
 \label{eq:central-lcm}
\end{gather}
For $1\leq i\leq3$ one has
\begin{equation}\label{eq:central-primitive-type}
 (m_i,h_i)=
 \left(
  \frac{p-1}{\gcd(p-1,r_i-1)},
  \frac{r_i-1}{\gcd(p-1,r_i-1)}
 \right).
\end{equation}
If $g>0$, then
\begin{equation}\label{eq:central-new-type}
 (m_4,h_4)=
 \left(
  \frac{p-1}{\gcd(p-1,2g)},
  \frac{p-1+2g}{\gcd(p-1,2g)}
 \right).
\end{equation}
\end{lemma}

\begin{proof}
Equation \eqref{eq:three-RH} gives
\[
 \sum_{i=1}^{3}A_i
 =\sum_{i=1}^{3}(r_i-1)=p-1-2g.
\]
For $g=0$ this is $\sum_{j\in J}A_j=p-1$; for $g>0$ the additional
definition $A_4=2g$ gives the same equality.  Since $A_j=Da_j$ and
$p-1=DM$, division by $D$ gives $\sum_j a_j=M$.

For $1\leq i\leq3$, the inequalities $2\leq r_i\leq p-1$ give
$0<A_i<p-1$.  If $g>0$, then \eqref{eq:genus-bound} gives
$0<A_4=2g\leq p-5<p-1$.  Division by the positive integer $D$ yields
$0<a_j<M$, and hence $M>1$.

Let $q=\gcd(M,(a_j)_{j\in J})$.  Then $Dq$ divides $DM=p-1$ and every
$Da_j=A_j$.  The definition of $D$ implies $Dq\mid D$, so $q=1$.  This
proves \eqref{eq:central-bounds}.  Moreover,
\[
 \gcd(p-1,A_j)=\gcd(DM,Da_j)=D\gcd(M,a_j)=Dc_j.
\]
Dividing this equality into $p-1$ and $A_j$ gives
\[
 m_j=\frac{p-1}{\gcd(p-1,A_j)},\qquad
 \widetilde a_j=\frac{A_j}{\gcd(p-1,A_j)}.
\]
These fractions are reduced, which proves
\eqref{eq:central-fractions}, \eqref{eq:central-primitive-type}, and
\eqref{eq:central-new-type}.

It remains to prove \eqref{eq:central-lcm}.  Fix a prime $\ell$ and write
$b=v_\ell(M)$.  Then
\[
 v_\ell(m_j)=b-\min\{b,v_\ell(a_j)\}.
\]
Taking the maximum over $j$ gives
\begin{align*}
 v_\ell\left(\operatorname{lcm}_{j\in J}m_j\right)
 &=b-\min\bigl\{b,\min_{j\in J}v_\ell(a_j)\bigr\}\\
 &=b-v_\ell\bigl(\gcd(M,(a_j)_{j\in J})\bigr)=b.
\end{align*}
Thus the least common multiple and $M$ have the same valuation at every
prime and are equal.
\end{proof}

Fix a regular $p$-cycle subgroup $P\leq S_p$.  Let $H_0$ be the unique
multiplier subgroup of $\mathbf F_p^\times$ of order $M$, regarded as the
standard complement in
$N_{S_p}(P)=P\rtimes\mathbf F_p^\times$.  Let
\[
 \chi_0:H_0\hookrightarrow\operatorname{Aut}(P)=\mathbf F_p^\times
\]
be its actual conjugation character.

\subsection{Genus zero}
Assume $g=0$.  Let $Z_0$ be the smooth projective normalization of
\begin{equation}\label{eq:genus-zero-kummer}
 z^M=x^{a_1}(x-1)^{a_2},
\end{equation}
and let $H_0$ act by fixing $x$ and by
$h^*z=\chi_0(h)z$.

\begin{lemma}\label{lem:genus-zero-connected}
The curve $Z_0$ is connected and the map $Z_0\to\mathbf P^1_k$ has degree
$M$.
\end{lemma}

\begin{proof}
Put $F=x^{a_1}(x-1)^{a_2}$.  If $F$ were an $\ell$th power in $k(x)$ for a
prime $\ell\mid M$, every valuation of $F$ would be divisible by $\ell$.
The valuations at $0$, $1$, and infinity are
\[
 a_1,\qquad a_2,\qquad -(a_1+a_2)=a_3-M.
\]
Thus $\ell$ would divide $M,a_1,a_2,a_3$, contradicting
\eqref{eq:central-bounds}.  The Kummer irreducibility criterion therefore
shows that $T^M-F$ is irreducible over $k(x)$.  Its function-field extension
has degree $M$, and its smooth projective normalization is connected.
\end{proof}

Define
\begin{equation}\label{eq:genus-zero-form}
 \omega=z\frac{dx}{x(x-1)}.
\end{equation}

\begin{lemma}[Divisor of the genus-zero differential]
\label{lem:genus-zero-divisor}
If $P_i$ is the reduced fiber above $0$, $1$, or infinity according as
$i=1,2,3$, then
\[
 \operatorname{div}(\omega)
 =\sum_{i=1}^3(\widetilde a_i-1)P_i.
 \label{eq:genus-zero-divisor}
\]
The differential is nonzero, regular, and belongs to the
$\chi_0$-eigenspace.
\end{lemma}

\begin{proof}
Lemma \ref{lem:kummer-critical-type} gives $c_i$ points above the $i$th
branch point, each of ramification index $m_i$.  Let $Q$ lie above $0$ and
choose its local parameter $s$.  Then
\[
 \operatorname{ord}_Q(x)=m_1,\quad
 \operatorname{ord}_Q(z)=\widetilde a_1,\quad
 \operatorname{ord}_Q(dx)=m_1-1,
\]
while $x-1$ is a unit.  Hence
$\operatorname{ord}_Q(\omega)=\widetilde a_1-1$.  Above $1$, the same
calculation with $x-1$ in place of $x$ gives
$\operatorname{ord}_Q(\omega)=\widetilde a_2-1$.

Above infinity put $t=x^{-1}$ and $z_3=tz$.  Since
$a_1+a_2+a_3=M$, equation \eqref{eq:genus-zero-kummer} becomes
\[
 z_3^M=t^{a_3}(1-t)^{a_2},
 \qquad
 \omega=-\frac{z_3\,dt}{t(1-t)}.
\]
At a point $Q$ above infinity,
\[
 \operatorname{ord}_Q(t)=m_3,\quad
 \operatorname{ord}_Q(z_3)=\widetilde a_3,\quad
 \operatorname{ord}_Q(dt)=m_3-1.
\]
It follows that $\operatorname{ord}_Q(\omega)=\widetilde a_3-1$.  At any
other point, the covering is \'etale, $z,x,x-1$ are units, and a translate
of $x$ is a local parameter; hence $\omega$ has order zero.  This proves
the divisor formula.  Its coefficients are nonnegative, so $\omega$ is
regular.  Finally $h^*z=\chi_0(h)z$ while $h$ fixes $x$, and therefore
$h^*\omega=\chi_0(h)\omega$.
\end{proof}

\begin{proposition}[Central datum in genus zero]
\label{prop:central-genus-zero}
There is a normalized logarithmic special deformation datum of type
$(H_0,\chi_0)$ whose three critical points are $0,1,\infty$ and whose
invariants are $\sigma_i=(r_i-1)/(p-1)$.
\end{proposition}

\begin{proof}
Here $r=3$, so Theorem \ref{thm:wewers-cyclic}(1) says that the
$\chi_0$-eigenspace is one-dimensional and Cartier is bijective on it.
Lemma \ref{lem:genus-zero-divisor} gives a nonzero differential $\omega$ in
this line.  Lemma \ref{lem:cartier-rescaling} gives $\lambda\in k^\times$
with $\mathcal C(\lambda\omega)=\lambda\omega$.  Multiplication by
$\lambda$ changes neither the divisor nor the character, so
$(Z_0,\lambda\omega)$ satisfies Definition
\ref{def:cyclic-degeneration} with $\nu_1=\nu_2=\nu_3=0$.  Lemma
\ref{lem:degeneration-to-deformation} makes it a normalized logarithmic
special deformation datum.  Lemmas \ref{lem:kummer-critical-type} and
\ref{lem:central-arithmetic} give
\[
 \sigma_i=\frac{a_i}{M}=\frac{A_i}{p-1}
 =\frac{r_i-1}{p-1}.
\]
\end{proof}

\subsection{Positive genus}
Assume $g>0$.

\begin{proposition}[Central datum in positive genus]
\label{prop:central-positive-genus}
There is a normalized logarithmic special deformation datum of type
$(H_0,\chi_0)$ with three primitive critical points $0,1,\infty$ having
invariants
\[
 \sigma_i=\frac{r_i-1}{p-1}\qquad(1\leq i\leq3),
\]
and one new critical point having invariant
\[
 \sigma_4=1+\frac{2g}{p-1}.
\]
\end{proposition}

\begin{proof}
Reorder the four indices by putting
\[
 (b_1,b_2,b_3,b_4)=(a_4,a_1,a_2,a_3),\qquad
 (\epsilon_1,\epsilon_2,\epsilon_3,\epsilon_4)=(1,0,0,0).
\]
Lemma \ref{lem:central-arithmetic} verifies the hypotheses of Theorem
\ref{thm:wewers-cyclic}(2).  In the notation of that theorem,
$\alpha=(p-1)/M=D$.  The number of equivalence classes is
\[
 \alpha b_1-1=Da_4-1=A_4-1=2g-1.
\]
Since $g>0$, this integer is at least one.  Choose one representative with
ordered branch points
$\tau'_1,\tau'_2,\tau'_3,\tau'_4$.  The point $\tau'_1$ has
$\epsilon_1=1$; the other three have $\epsilon_j=0$.

There is a unique projective transformation $u\in\operatorname{PGL}_2(k)$
such that
\[
 u(\tau'_2)=0,\qquad u(\tau'_3)=1,\qquad
 u(\tau'_4)=\infty.
\]
Transport the cyclic cover, its $H_0$-action, and its differential along
$u$.  Pullback by an isomorphism commutes with taking divisors and with the
Cartier operator, so the transported objects still satisfy
\eqref{eq:degeneration-divisor} and \eqref{eq:degeneration-cartier} with
the same character.  Rename the transported primitive points as
$0,1,\infty$ and the transported new point as $\tau_4$.  Lemma
\ref{lem:degeneration-to-deformation} gives a normalized logarithmic special
deformation datum.  Finally, Lemmas \ref{lem:kummer-critical-type} and
\ref{lem:central-arithmetic} give
\[
 \sigma_i=\frac{a_i}{M}=\frac{r_i-1}{p-1}\qquad(1\leq i\leq3),
 \qquad
 \sigma_4=1+\frac{a_4}{M}=1+\frac{2g}{p-1}.
\]
\end{proof}

\section{Assembly in one permutation group and lifting}
\label{sec:assembly}

\begin{lemma}[Simultaneous group and character alignment]
\label{lem:group-alignment}
For $j\in J$, the natural tail groups may be embedded as subgroups
$G_j\leq S_p$ so that
they contain the same regular subgroup $P$, their inertia groups are
$P\rtimes H_j$ with $H_j\leq H_0$, and their conjugation characters are
the restrictions of the single character
$\chi_0:H_0\to\operatorname{Aut}(P)$.  With
\[
 G_0=P\rtimes H_0,\qquad
 G=\langle G_0,G_j:j\in J\rangle,
\]
the group $G$ is transitive and $p\Vert |G|$.
\end{lemma}

\begin{proof}
For $1\leq i\leq3$, comparison of
\eqref{eq:primitive-type} with
\eqref{eq:central-primitive-type} shows that the order of the tame
complement of the $i$th primitive tail is the central integer $m_i$.
When $g>0$, comparison of \eqref{eq:new-type} with
\eqref{eq:central-new-type} gives the same assertion for the new tail.
We use $m_j$ for this common integer.  By
\eqref{eq:central-lcm}, every $m_j$ divides $M$.

The wild inertia subgroup of each natural tail group is generated by a
$p$-cycle.  Conjugating its embedding in $S_p$ identifies that subgroup
with the fixed $P$.  The full inertia group normalizes $P$, so its tame
complement maps injectively to
$N_{S_p}(P)/P\simeq\mathbf F_p^\times$.  The cyclic group
$H_0$ has a unique subgroup of order $m_j$; denote it by $H_j$.  Lemma
\ref{lem:normalizer} permits a further conjugation by an element of $P$ that
identifies the tail complement with this literal subgroup $H_j$.  Transport
the marked point $\eta_j$ and the chosen complement together with the
cover; this changes neither the pointed-tail isomorphism class nor its
integer type.

For $h\in H_j$, both the tail character and $\chi_0|_{H_j}$ are defined by
the same equality
\[
 hxh^{-1}=x^{\chi_0(h)}\qquad(x\in P).
\]
Thus the characters agree element by element.

The group $G$ contains the transitive subgroup $P$, and is therefore
transitive.  It is a subgroup of $S_p$ and contains an element of order $p$.
Since $v_p(p!)=1$, one has $p\mid |G|$ and $p^2\nmid |G|$.
\end{proof}

\begin{lemma}[Stabilizers on the central cyclic cover]
\label{lem:central-stabilizers}
Let $\xi_j\in Z_0$ lie over the critical point $\tau_j$.  Then
\[
 \operatorname{Stab}_{H_0}(\xi_j)=H_j,
 \qquad I_{\eta_j}=P\rtimes H_j,
 \qquad H_0\cap I_{\eta_j}=H_j.
\]
\end{lemma}

\begin{proof}
The inertia order of the cyclic cover at $\tau_j$ is
$M/\gcd(M,a_j)=m_j$.  Since $H_0$ is cyclic, it has exactly one subgroup of
that order, namely $H_j$.  This proves the first equality.  The second
equality is part of the pointed-tail construction after Lemma
\ref{lem:group-alignment}.  For the third, let
$x\in H_0\cap(P\rtimes H_j)$.  In the semidirect product
$P\rtimes H_0$, the element $x$ has two expressions, namely $x=1\cdot x$
and $x=q h$ with $q\in P$ and $h\in H_j$.  Uniqueness of a
semidirect-product expression gives $q=1$ and $x=h\in H_j$.  This proves
$H_0\cap I_{\eta_j}\subseteq H_j$; the reverse inclusion is immediate.

\end{proof}

\begin{proposition}[Construction of the special $G$-deformation datum]
\label{prop:assembled-special-datum}
Let $G_0$, the tail groups $G_j$, and $G$ be chosen as in Lemma
\ref{lem:group-alignment}.  If $g=0$, use the central datum of Proposition
\ref{prop:central-genus-zero} and the three primitive tails of Lemma
\ref{lem:primitive-tail}.  If $g>0$, use the central datum of Proposition
\ref{prop:central-positive-genus}, the same three primitive tails, and the
new tail of Lemma \ref{lem:newtail}.  Choose a point $\xi_j$ of the central
cyclic cover above every critical point $\tau_j$.  These objects form a
special $G$-deformation datum in the sense of Definition
\ref{def:special-G-datum}.  Its central datum is logarithmic, has no wild
critical point, and has exactly three primitive critical points.  Moreover,
$G\leq S_p$ is transitive and $p\Vert |G|$.
\end{proposition}

\begin{proof}
We verify the clauses of Definition \ref{def:special-G-datum} in their
stated order.  If $g=0$, Proposition \ref{prop:central-genus-zero} provides
a normalized logarithmic special deformation datum of type
$(H_0,\chi_0)$.  If $g>0$, Proposition
\ref{prop:central-positive-genus} provides such a datum.  In both cases,
for $1\leq i\leq3$ its invariant is
\[
 \sigma_i=\frac{r_i-1}{p-1}.
\]
The inequalities $2\leq r_i\leq p-1$ imply
$0<\sigma_i<1$.  If $g>0$, the fourth invariant is
\[
 \sigma_4=1+\frac{2g}{p-1}.
\]
Equation \eqref{eq:genus-bound} gives $0<2g<p-1$, and hence
$1<\sigma_4<2$.  Thus there are exactly three primitive points, there is
one new point precisely when $g>0$, and there is no wild point.

Lemma \ref{lem:primitive-tail} supplies a connected pointed primitive tail
for each of the first three points.  For those three points, Lemma
\ref{lem:central-arithmetic}, together with \eqref{eq:primitive-type},
shows that the tail type equals the central type.  When $g>0$, Lemma
\ref{lem:newtail} supplies the connected pointed new tail for the fourth
point, and \eqref{eq:new-type} together with
\eqref{eq:central-new-type} gives the equality of its two integer types.

Lemma \ref{lem:group-alignment} embeds all these tail groups in one literal
group $G\leq S_p$, identifies their wild inertia subgroup with the same
$P$, and identifies their complements with subgroups $H_j\leq H_0$ having
the actual character $\chi_0|_{H_j}$.  By definition of $G$ in that lemma,
\[
 G=\langle G_0,G_j:j\in J\rangle.
\]
Lemma \ref{lem:central-stabilizers} gives, for every $j\in J$,
\[
 I_{\eta_j}=P\rtimes H_j,
 \qquad
 H_j=\operatorname{Stab}_{H_0}(\xi_j)=H_0\cap I_{\eta_j}.
\]
These are all the compatibility equalities in Definition
\ref{def:special-G-datum}.  Finally, Lemma \ref{lem:group-alignment} proves
that $G$ is transitive and that $p\Vert |G|$.
\end{proof}

\begin{lemma}[Geometric connectedness and profiles under change of field]
\label{lem:char-zero-base-change}
Let $L$ be a field of characteristic zero, let
$f:X\to\mathbf P^1_L$ be a finite morphism from a smooth projective curve,
and let $b_1,\ldots,b_s\in\mathbf P^1(L)$ be distinct.  Suppose that, after
base change to one algebraically closed extension $\Omega/L$, the curve
$X_\Omega$ is connected, the map $f_\Omega$ has degree $d$, is \'etale away
from the $b_i$, and has profile $\Lambda_i$ above $b_i$.  Then $X$ is
geometrically connected.  For every algebraically closed extension
$\Omega'/L$, the map $f_{\Omega'}$ has degree $d$, is \'etale away from the
$b_i$, and has profile $\Lambda_i$ above $b_i$.
\end{lemma}

\begin{proof}
Fix an algebraic closure $\overline L$ of $L$ inside $\Omega$.  If
$X_{\overline L}$ were disconnected, it would admit a nontrivial open and
closed decomposition.  Base change of that decomposition along the
faithfully flat extension $\Omega/\overline L$ would disconnect
$X_\Omega$.  Therefore $X_{\overline L}$ is connected, which is the
definition of geometric connectedness.

Because $X$ is smooth and geometrically connected, it is integral.  A
finite dominant morphism from an integral curve to the regular curve
$\mathbf P^1_L$ is finite locally free: at a point of the target, its
direct image is a finite torsion-free module over a discrete valuation
ring, and every such module is free.  The rank can be computed after the
faithfully flat extension $\Omega/L$, where it is $d$.  Thus every base
change of $f$ has degree $d$.

Put $U=\mathbf P^1_L-\{b_1,\ldots,b_s\}$.  The morphism
$f^{-1}(U)\to U$ becomes \'etale after the faithfully flat extension
$\Omega/L$.  Since being \'etale is local for the faithfully flat topology
on source and target, it is already \'etale over $L$.  It remains \'etale
after every base change, so no additional branch value can occur over
$\Omega'$.

For each $i$, let
\[
 \mathcal A_i=H^0\bigl(X\times_{\mathbf P^1_L}\operatorname{Spec}L,
              \mathcal O\bigr),
\]
where $\operatorname{Spec}L\to\mathbf P^1_L$ is the point $b_i$.  This is
a finite $L$-algebra of dimension $d$.  Over $\overline L$, it has a
canonical product decomposition into its Artinian local factors,
\[
 \mathcal A_i\otimes_L\overline L
 =\prod_{x\in f^{-1}(b_i)(\overline L)}\mathcal A_{i,x}.
\]
The ramification index at $x$ equals
$\dim_{\overline L}\mathcal A_{i,x}$: after choosing uniformizers on the two smooth
curves, the completed local map has the form
$t=u s^{e_x}$ with $u$ a unit, and the fiber local ring has length $e_x$.
Choose an $L$-embedding $\overline L\hookrightarrow\Omega'$; it exists
because $\Omega'$ is algebraically closed.  If $\mathcal B$ is one of the
Artinian local factors $\mathcal A_{i,x}$ and $\mathfrak m$ is its maximal
ideal, then $\mathcal B/\mathfrak m=\overline L$ and $\mathfrak m$ is
nilpotent.  For every algebraically closed extension $E/\overline L$, the
ideal $\mathfrak m\otimes_{\overline L}E$ is nilpotent and the quotient of
$\mathcal B\otimes_{\overline L}E$ by this ideal is $E$.  Hence
$\mathcal B\otimes_{\overline L}E$ is again local, and its dimension over
$E$ equals $\dim_{\overline L}\mathcal B$.  Apply this first to
$E=\Omega$ and then to $E=\Omega'$.  The geometric points above $b_i$ and
the lengths of their local fiber rings are therefore indexed by the same
factors and have the same multiset over both fields.  Since this multiset is
$\Lambda_i$ over $\Omega$, it is $\Lambda_i$ over $\Omega'$.
\end{proof}

\begin{theorem}[The residual three-point case]\label{thm:residual}
Let $p$ be odd.  Let $\Lambda_1,\Lambda_2,\Lambda_3\vdash p$ be nonidentity
partitions, none equal to $[p]$, and suppose that
\eqref{eq:three-RH} holds.  Then there is a connected degree-$p$ cover of
$\mathbf P^1(\mathbf C)$ with these three profiles.
\end{theorem}

\begin{proof}
Proposition \ref{prop:assembled-special-datum} verifies every hypothesis of
Theorem \ref{prop:interface}.  Hence there is a connected three-point
$G$-cover over an algebraically closed field $K$ of characteristic zero,
\[
 F:Y\longrightarrow\mathbf P^1_K.
\]
After composing $F$ with a projective automorphism of its target, label its
three branch values by $0$, $1$, and $\infty$ in the order of the three
primitive tails.

Fix a letter $\alpha$ and let $G_\alpha$ be its stabilizer.  The natural map
$G/G_\alpha\to\{1,\ldots,p\}$, $gG_\alpha\mapsto g(\alpha)$, is a
$G$-equivariant bijection.  Therefore $[G:G_\alpha]=p$, and
\[
 f_K:X=Y/G_\alpha\longrightarrow\mathbf P^1_K
\]
has degree $p$.  The smooth curve $Y$ is connected and hence integral.  Its
finite quotient by $G_\alpha$ is integral; after taking the smooth projective
model, which does not alter its function field or the finite map to
$\mathbf P^1_K$, the curve $X$ is smooth, projective, and connected.

Let $T_i\leq G_i$ be the tame inertia subgroup at the finite branch point
of the $i$th primitive tail, and choose a generator $t_i\in T_i$.  Lemma
\ref{lem:primitive-tail} states that $t_i$ has cycle lengths
$e_{i1},\ldots,e_{ir_i}$ in the natural $p$-letter action.  Wewers's formal
construction uses the unique tame lift of the induced affine tail cover.
Lemma \ref{lem:tame-profile} therefore preserves the conjugacy class of the
cyclic inertia subgroup and preserves the cycle partition of a generator
on every $G$-set.  Under the displayed $G$-equivariant bijection, its action
on $G/G_\alpha$ has the same cycle lengths as its action on the $p$ letters.
Lemma \ref{lem:quotient-profile} now gives profile $\Lambda_i$ for $f_K$
over the $i$th branch value.  The affine part of the new tail, when present,
is \'etale; Lemma \ref{lem:tame-profile} shows that its lift is also \'etale.
Theorem \ref{prop:interface} has exactly the three primitive horizontal
branch sections.  Passing to an intermediate quotient cannot introduce a
branch point, so $f_K$ has no other branch value.

The total ramification contribution is
\[
 \sum_{i=1}^{3}\sum_{j=1}^{r_i}(e_{ij}-1)
 =3p-(r_1+r_2+r_3)=2p-2+2g.
\]
Riemann--Hurwitz gives $2g(X)-2=-2p+(2p-2+2g)=2g-2$,
and hence $g(X)=g$.

The curve $X$, the finite morphism $f_K$, and the three labeled target
points are objects of finite presentation.  Consequently, they descend to
a finitely generated subfield $L\subset K$ over $\mathbf Q$:
\[
 f_L:X_L\longrightarrow\mathbf P^1_L,
 \qquad
 f_L\otimes_L K\simeq f_K.
\]
Finiteness and smoothness descend along the faithfully flat extension
$K/L$; projectivity follows because $f_L$ is finite and its target is
projective.  Apply Lemma \ref{lem:char-zero-base-change} with
$\Omega=K$ and $(b_1,b_2,b_3)=(0,1,\infty)$.  It shows that $X_L$ is
geometrically connected and that degree, absence of other branch values,
and all three complete profiles persist after every algebraically closed
base change from $L$.

Choose a transcendence basis $t_1,\ldots,t_n$ of $L/\mathbf Q$.  There are
algebraically independent complex numbers $z_1,\ldots,z_n$, so the
assignment $t_i\mapsto z_i$ embeds
$\mathbf Q(t_1,\ldots,t_n)$ into $\mathbf C$.  The extension of $L$ over
this rational function field is finite.  Since $\mathbf C$ is algebraically
closed, the embedding extends to an embedding $L\hookrightarrow\mathbf C$.
Lemma \ref{lem:char-zero-base-change}, now with $\Omega'=\mathbf C$, proves
that $f_L\otimes_L\mathbf C$ is connected of degree $p$, is branched only
at $0,1,\infty$, and has profiles
$\Lambda_1,\Lambda_2,\Lambda_3$.  This is the required complex cover.
\end{proof}

\section{Completion of the proof and boundary cases}

Theorem \ref{thm:residual} and Lemma \ref{lem:EKS} prove Theorem
\ref{thm:main} for odd primes.  It remains to treat $p=2$.  The only
nonidentity partition of $2$ is $[2]$, and its ramification defect is one.
If the branch datum has $k$ branch points, Riemann--Hurwitz is
\[
 2-2g=4-k,
\]
so $k=2+2g$ is even.  Assign the transposition $\tau=(12)$ to every branch
point.  Then
\[
 \tau^k=(\tau^2)^{k/2}=1,
\]
and the subgroup generated by the assigned permutations is
$\langle\tau\rangle=S_2$, which is transitive on two letters.  The Hurwitz
monodromy criterion, Theorem \ref{thm:hurwitz-criterion}, produces the
required connected cover.  This completes the proof in every prime degree.

\section{Applications to Hurwitz spaces and relative Gromov--Witten theory}
\label{sec:applications}

We first record the consequence of Theorem~\ref{thm:main} for ordinary
connected Hurwitz numbers.  We then place the corresponding Hurwitz loci in
relative stable-map spaces.  The last part of the section uses the
Gromov--Witten/Hurwitz correspondence to obtain a strict positivity theorem
for standard relative descendent invariants.

\subsection{Full support of the connected Hurwitz potential}

For a partition $\Lambda\vdash d$, let $C_\Lambda\subset S_d$ be the
corresponding conjugacy class.  For ordered partitions
$\boldsymbol\Lambda=(\Lambda_1,\ldots,\Lambda_k)$ of $d$, define the
connected Hurwitz number by
\begin{equation}\label{eq:connected-hurwitz-number}
 \begin{split}
 H^\circ_{g,d}(\Lambda_1,\ldots,\Lambda_k)
 =\frac{1}{d!}\#\biggl\{(\sigma_1,\ldots,\sigma_k):\ &
 \sigma_i\in C_{\Lambda_i},\quad \sigma_1\cdots\sigma_k=1,\\[-2pt]
 &\langle\sigma_1,\ldots,\sigma_k\rangle
 \text{ is transitive}\biggr\}.
 \end{split}
\end{equation}
The subscript $g$ means that the Riemann--Hurwitz equality
\begin{equation}\label{eq:application-RH}
 \sum_{i=1}^k\bigl(d-\ell(\Lambda_i)\bigr)=2d-2+2g
\end{equation}
holds; the number is set equal to zero when this equality does not hold.
Formula \eqref{eq:connected-hurwitz-number} is the groupoid cardinality of
the covers with fixed ordered branch values.  Indeed, simultaneous
conjugation by $S_d$ changes the labeling of a generic fiber, and the
stabilizer of a tuple is the automorphism group of the associated cover.

\begin{corollary}[Prime-degree support saturation]
\label{cor:hurwitz-support}
Let $p$ be prime, let $g\geq0$, and let
$\Lambda_1,\ldots,\Lambda_k\vdash p$ be nonidentity partitions.  Then
\begin{equation}\label{eq:hurwitz-support-iff}
 H^\circ_{g,p}(\Lambda_1,\ldots,\Lambda_k)>0
 \quad\Longleftrightarrow\quad
 \sum_{i=1}^k\bigl(p-\ell(\Lambda_i)\bigr)=2p-2+2g.
\end{equation}
\end{corollary}

\begin{proof}
If the Hurwitz number is positive, one of the tuples in
\eqref{eq:connected-hurwitz-number} exists.  The associated cover satisfies
Riemann--Hurwitz, and hence the equality on the right of
\eqref{eq:hurwitz-support-iff}.  Conversely, that equality says precisely
that the ordered family is a compatible branch datum of degree $p$ and
genus $g$.  Theorem~\ref{thm:main} realizes it by a connected cover.  The
Hurwitz monodromy criterion gives a transitive tuple, so the set counted in
\eqref{eq:connected-hurwitz-number} is nonempty.
\end{proof}

The support statement may be expressed as an identity for a generating
series without losing the ordering of the branch values.  For
$\Lambda=(1^{m_1}2^{m_2}\cdots)$, put
\[
 \mathfrak z(\Lambda)=\prod_{a\geq1}a^{m_a}m_a!.
\]
For each branch-value slot $i$, introduce an independent family of
variables $P^{(i)}_\Lambda$, indexed by the nonidentity partitions of $p$,
and set
\begin{equation}\label{eq:prime-hurwitz-potential}
 \mathcal F^\circ_{p,k}
 =\sum_{g\geq0}\ \sum_{\substack{\Lambda_1,\ldots,\Lambda_k\vdash p\\
                                  \Lambda_i\neq1^p}}
 H^\circ_{g,p}(\Lambda_1,\ldots,\Lambda_k)
 u^{2g-2}\prod_{i=1}^k
 \frac{P^{(i)}_{\Lambda_i}}{\mathfrak z(\Lambda_i)}.
\end{equation}
The alphabets are independent because a passport remembers which profile
belongs to which branch value.  Corollary~\ref{cor:hurwitz-support} gives
the exact equality
\begin{equation}\label{eq:potential-support}
 \operatorname{Supp}(\mathcal F^\circ_{p,k})
 =\left\{\begin{array}{c|c}
  (g,\Lambda_1,\ldots,\Lambda_k)&
  \begin{array}{c}
   g\geq0,\quad \Lambda_i\vdash p,\quad \Lambda_i\neq1^p,\\
   \displaystyle\sum_{i=1}^k\bigl(p-\ell(\Lambda_i)\bigr)=2p-2+2g
  \end{array}
 \end{array}\right\},
\end{equation}
and every coefficient on this support is a positive rational number.
Identity partitions may be allowed as dummy unramified slots, provided
they are deleted before the support is recorded.

\subsection{Hurwitz cycles over the moduli of branch values}

Let
\[
 U_k=\operatorname{Conf}^{\mathrm{ord}}_k(\mathbf P^1)
 =\{(b_1,\ldots,b_k)\in(\mathbf P^1)^k:b_i\neq b_j\text{ for }i\neq j\}
\]
be the ordered configuration space.  Let
$\mathcal H^\circ_{\boldsymbol\Lambda}$ be the complex Hurwitz stack of
connected degree-$p$ covers with ordered branch values and with profile
$\Lambda_i$ above the $i$th value.  There is a branch morphism
\[
 \operatorname{br}:\mathcal H^\circ_{\boldsymbol\Lambda}\longrightarrow U_k.
\]
The analytic and algebraic constructions of this morphism, together with
the admissible-cover compactification used below, are reviewed in
\cite[Section~2.5, Proposition~3.2, Theorem~4.11,
Remark~4.15(i), and Section~5]{RomagnyWewers}.  The cited statements are
formulated first for Galois covers; Section~2.5 explains the corresponding
non-Galois and ordered variants used here.
For the compactification in the non-Galois case we use the stable Hurwitz
stack of \cite[Sections~6.3 and~6.6, especially Theorem~6.3.1 and
Proposition~6.6.12]{BertinRomagny}; we take its ordered base change and the
open-and-closed locus with the prescribed local profiles.

\begin{proposition}[Full support over the branch configuration space]
\label{prop:branch-map-full-support}
Assume that $\boldsymbol\Lambda$ satisfies
\eqref{eq:application-RH}.  The coarse branch morphism is finite, \'etale,
and surjective.  The Deligne--Mumford stack morphism
$\operatorname{br}$ is proper with finite geometric fibers and is
analytically an orbifold covering.  Its stack-theoretic weighted degree is
\[
 h^\circ_{\boldsymbol\Lambda}
 =H^\circ_{g,p}(\Lambda_1,\ldots,\Lambda_k)>0,
\]
and, with rational coefficients,
\begin{equation}\label{eq:branch-cycle-pushforward}
 \operatorname{br}_*[\mathcal H^\circ_{\boldsymbol\Lambda}]
 =h^\circ_{\boldsymbol\Lambda}[U_k].
\end{equation}
Here \eqref{eq:branch-cycle-pushforward} is an equality in
$A_k(U_k)_{\mathbf Q}$.  If $k\geq3$, put
\[
 \mathcal H^{\mathrm{red}}_{\boldsymbol\Lambda}
 =\bigl[\mathcal H^\circ_{\boldsymbol\Lambda}/
        \operatorname{PGL}_2\bigr].
\]
Its coarse branch morphism to $M_{0,k}$ is finite and surjective, while the
stack morphism has weighted degree $h^\circ_{\boldsymbol\Lambda}$.
Decompose $\mathcal H^{\mathrm{red}}_{\boldsymbol\Lambda}$ into its
finitely many open-and-closed strata indexed by the non-Galois monodromy
types $\mathbf m=(G,H,\xi)$ that occur.  For every such type and every
irreducible component, take its reduced closure in the corresponding
ordered Bertin--Romagny stable Hurwitz stack and then take the
stack-theoretic normalization.  Define
$\overline{\mathcal H}^{\mathrm{dom}}_{\boldsymbol\Lambda}$ to be the
disjoint union of these normalizations.  Thus all monodromy types occurring
on the open Hurwitz stack are included, and no component disjoint from that
open stack is added.  Then the proper branch morphism
$\overline{\operatorname{br}}$ satisfies
\begin{equation}\label{eq:compactified-branch-cycle}
 \overline{\operatorname{br}}_*
 [\overline{\mathcal H}^{\mathrm{dom}}_{\boldsymbol\Lambda}]
 =h^\circ_{\boldsymbol\Lambda}[\overline M_{0,k}]
 \quad\text{in }A_{k-3}(\overline M_{0,k})_{\mathbf Q}.
\end{equation}
\end{proposition}

\begin{proof}
Fix a point of $U_k$ and a system of peripheral loops.  By
Theorem~\ref{thm:hurwitz-criterion}, the fiber of the branch morphism is the
finite groupoid of transitive tuples of the prescribed cycle types, modulo
simultaneous conjugation.  When the branch values move in a sufficiently
small simply connected open subset of $U_k$, the loops move by isotopy and
the tuple is transported uniquely.  Thus the stack morphism is locally an
orbifold covering, and its coarse morphism is locally a finite covering.
Analytically continuing around a loop in $U_k$ gives the usual braid action
on the finite set of Nielsen tuples.

Corollary~\ref{cor:hurwitz-support} gives at least one tuple.  The same tuple,
together with peripheral loops around any other ordered configuration,
defines a cover by Theorem~\ref{thm:hurwitz-criterion}.  Hence every fiber is
nonempty and the branch morphism is surjective.  Its stack-theoretic degree
is the sum of $1/|\operatorname{Aut}(f)|$ over one fiber, which is exactly
\eqref{eq:connected-hurwitz-number}.  This proves
\eqref{eq:branch-cycle-pushforward}.

For $k\geq3$, an automorphism of $\mathbf P^1$ fixing the ordered branch
configuration is the identity, so quotienting source and target by
$\operatorname{PGL}_2$ preserves the generic stack-theoretic weighted
degree.  By construction,
$\overline{\mathcal H}^{\mathrm{dom}}_{\boldsymbol\Lambda}$ is proper and
generically finite of that degree over $\overline M_{0,k}$.  Proper
pushforward of the stack-theoretic, ordinary non-virtual fundamental class of this
Deligne--Mumford stack therefore gives its generic weighted degree times
the fundamental class of $\overline M_{0,k}$, proving
\eqref{eq:compactified-branch-cycle} in rational Chow theory.
\end{proof}

\subsection{The relative stable-map stratum}

Fix $(b_1,\ldots,b_k)\in U_k$, set $D=b_1+\cdots+b_k$, and write
$\Lambda_i=(e_{i1},\ldots,e_{ir_i})$.  Let
\[
 R=\sum_{i=1}^k r_i.
\]
Consider relative stable maps to $(\mathbf P^1,D)$ of degree $p$ and genus
$g$, with $R$ relative markings whose contact orders over $b_i$ are
$e_{i1},\ldots,e_{ir_i}$.  Expanded relative stable maps and their proper
moduli spaces are constructed in \cite{JunLiStable}; the stationary relative
theory in the notation used below is recalled in
\cite[Section~1.2]{OkounkovPandharipande}.

\begin{proposition}[Interior nonemptiness at expected dimension zero]
\label{prop:relative-interior}
If $p$ is prime and the profiles satisfy
\eqref{eq:application-RH}, then the relative stable-map space has virtual
complex dimension zero and contains a point represented by a finite map
from a smooth connected curve.  If the $k$ target points vary, the
corresponding open Hurwitz locus has dimension $k$; after quotienting by
$\operatorname{PGL}_2$ for $k\geq3$, it has dimension $k-3$ and dominates
$M_{0,k}$.
\end{proposition}

\begin{proof}
The logarithmic tangent bundle has degree
\[
 \deg T_{\mathbf P^1}(-\log D)=2-k.
\]
The expected-dimension formula for maps to the logarithmic curve
$(\mathbf P^1,D)$, with the $R$ relative markings included, gives
\begin{align*}
 \operatorname{vdim}_{\mathbf C}
 &= (1-g)(1-3)+p\deg T_{\mathbf P^1}(-\log D)+R\\
 &=2g-2+p(2-k)+R.
\end{align*}
Since
\[
 \sum_{i=1}^k\bigl(p-r_i\bigr)=kp-R=2p-2+2g,
\]
the displayed virtual dimension is zero.  Theorem~\ref{thm:main} supplies a
connected cover with precisely the required contact orders.  It determines
a point in the open substack where the source is smooth and the stable map
is finite.  Allowing the ordered target points to move adds $k$ parameters,
and quotienting by the three-dimensional group $\operatorname{PGL}_2$
leaves dimension $k-3$.  Dominance is
Proposition~\ref{prop:branch-map-full-support}.
\end{proof}

\subsection{Relative virtual nonvanishing for every passport}

We use the partition convention for relative conditions: relative roots of
equal contact order are not individually labeled.  To make the
normalization explicit, write
\[
 G_{\boldsymbol\Lambda}
 =\prod_{i=1}^k\prod_{a\geq1}S_{m_a(\Lambda_i)},
 \qquad
 A_{\boldsymbol\Lambda}=|G_{\boldsymbol\Lambda}|
 =\prod_{i=1}^k\prod_{a\geq1}m_a(\Lambda_i)!.
\]
If $\overline{\mathcal M}^{\mathrm{rel,lab}}$ denotes Jun Li's stack with
the relative roots individually labeled, set
\[
 \overline{\mathcal M}^{\mathrm{rel,part}}
 =\bigl[\overline{\mathcal M}^{\mathrm{rel,lab}}/
        G_{\boldsymbol\Lambda}\bigr].
\]
Define the connected relative invariant
\begin{equation}\label{eq:relative-invariant-definition}
 \left\langle\Lambda_1,\ldots,\Lambda_k
 \right\rangle^{\circ,\mathbf P^1}_{g,p}
 =\int_{[\overline{\mathcal M}^{\mathrm{rel,part}}_{g}
   (\mathbf P^1,D;p,\boldsymbol\Lambda)]^{\mathrm{vir}}}1.
\end{equation}
The virtual class and the relative degeneration formula are constructed in
\cite{JunLiDegeneration}.  Proposition~\ref{prop:relative-interior} shows
that the integral has virtual dimension zero.

\begin{theorem}[Relative Gromov--Witten nonvanishing]
\label{thm:relative-GW-nonvanishing}
Let $p$ be prime and let
$\boldsymbol\Lambda=(\Lambda_1,\ldots,\Lambda_k)$ be a compatible ordered
family of nonidentity partitions of $p$.  Then
\begin{equation}\label{eq:relative-GW-equals-Hurwitz}
 \left\langle\Lambda_1,\ldots,\Lambda_k
 \right\rangle^{\circ,\mathbf P^1}_{g,p}
 =H^\circ_{g,p}(\Lambda_1,\ldots,\Lambda_k)>0.
\end{equation}
\end{theorem}

\begin{proof}
Okounkov--Pandharipande define the Gromov--Witten theory of a smooth target
curve relative to an arbitrary finite set of target points, with one
partition of the degree prescribed at each point
\cite[Section~1.2]{OkounkovPandharipande}.  Their relative
Gromov--Witten/Hurwitz correspondence is
\begin{multline}\label{eq:OP-relative-correspondence}
 \left\langle
   \prod_{a=1}^{n}\tau_{m_a}(\omega),
   \eta_1,\ldots,\eta_s
 \right\rangle^{\bullet X}_{d}\\
 =\frac1{\prod_{a=1}^{n}m_a!}
 H^{\bullet X}_{d}
 \left(\overline{(m_1+1)},\ldots,\overline{(m_n+1)},
       \eta_1,\ldots,\eta_s\right),
\end{multline}
where $\bullet$ denotes the disconnected theory
\cite[Theorem~1]{OkounkovPandharipande}.  Set $n=0$, take
$X=\mathbf P^1$, $d=p$, and $\eta_i=\Lambda_i$.  No completed cycle occurs,
so \eqref{eq:OP-relative-correspondence} is the equality of the
disconnected relative invariant and the disconnected Hurwitz number with
the prescribed profiles.

To extract the connected coefficient without merging the $k$ relative
conditions, introduce a separate alphabet
$\mathbf x^{(i)}=(x^{(i)}_1,x^{(i)}_2,\ldots)$ for every relative target
point and write $x^{(i)}_\Lambda=\prod_jx^{(i)}_{\lambda_j}$.  In the
unlabeled partition convention set
\begin{align*}
 Z^\bullet_{\mathrm{rel}}
 &=1+\sum_{d>0}q^d
   \sum_{\Lambda_1,\ldots,\Lambda_k\vdash d}
   \left(\prod_{i=1}^k x^{(i)}_{\Lambda_i}\right)
   \left\langle\Lambda_1,\ldots,\Lambda_k
   \right\rangle^{\bullet,\mathbf P^1}_{d},\\
 Z^\bullet_{\mathrm H}
 &=1+\sum_{d>0}q^d
   \sum_{\Lambda_1,\ldots,\Lambda_k\vdash d}
   \left(\prod_{i=1}^k x^{(i)}_{\Lambda_i}\right)
   H^\bullet_d(\Lambda_1,\ldots,\Lambda_k).
\end{align*}
The specialization $n=0$ of Theorem~1 gives
$Z^\bullet_{\mathrm{rel}}=Z^\bullet_{\mathrm H}$ in exactly these
partition variables.  The sums include identity partitions: a connected
component of a disconnected cover can be unramified over some relative
target point.  If $C=\coprod_{\alpha=1}^cC_\alpha$, the disconnected
arithmetic-genus convention of
\cite[Section~0.2]{OkounkovPandharipande} gives
\[
 2g(C)-2=\sum_{\alpha=1}^c\bigl(2g(C_\alpha)-2\bigr).
\]
Thus the genus selected by the zero-dimensional constraint is compatible
with disjoint union.

We now use the standard connected/disconnected exponential identity in
relative Gromov--Witten theory.  In the two-relative-point notation of
\cite[equations~(4.1)--(4.2)]{OkounkovPandharipande}, it states that the
free energy of connected relative invariants is the logarithm of the
disconnected partition function.  The same identity applies to any finite
set of relative target points: disjoint union concatenates the partition
at each target point, and the independent alphabets record those
concatenations separately.  This use of the exponential identity does not
assert that a fixed expanded-target moduli stack is literally a Cartesian
product of connected moduli stacks; it is the connected/disconnected
identity for the relative virtual theory used in the cited correspondence.
On the Hurwitz side the ordinary groupoid exponential formula gives the
same statement, including the quotient weights for repeated connected
components.  Consequently
\begin{align*}
 [q^p\prod_i x^{(i)}_{\Lambda_i}]
 \log Z^\bullet_{\mathrm{rel}}
 &=\left\langle\Lambda_1,\ldots,\Lambda_k
   \right\rangle^{\circ,\mathbf P^1}_{g,p},\\
 [q^p\prod_i x^{(i)}_{\Lambda_i}]
 \log Z^\bullet_{\mathrm H}
 &=H^\circ_{g,p}(\Lambda_1,\ldots,\Lambda_k).
\end{align*}
Equality of the two logarithms proves
\eqref{eq:relative-GW-equals-Hurwitz}; no additional
$\mathfrak z(\Lambda_i)$ or automorphism factor occurs because both sides
already use the same groupoid and partition normalization.
Corollary~\ref{cor:hurwitz-support} makes the common value strictly
positive.
\end{proof}

With individually labeled relative roots, the exact conversion is
\begin{equation}\label{eq:labeled-relative-conversion}
 \int_{[\overline{\mathcal M}^{\mathrm{rel,lab}}]^{\mathrm{vir}}}1
 =A_{\boldsymbol\Lambda}
  \left\langle\Lambda_1,\ldots,\Lambda_k
  \right\rangle^{\circ,\mathbf P^1}_{g,p}
 =A_{\boldsymbol\Lambda}
  H^\circ_{g,p}(\Lambda_1,\ldots,\Lambda_k).
\end{equation}
Indeed, the quotient map
\[
 q:\overline{\mathcal M}^{\mathrm{rel,lab}}
 \longrightarrow\overline{\mathcal M}^{\mathrm{rel,part}}
\]
is a finite \'etale $G_{\boldsymbol\Lambda}$-torsor of degree
$A_{\boldsymbol\Lambda}$.  The quotient obstruction theory pulls back
along $q$, and therefore
\[
 q_*[\overline{\mathcal M}^{\mathrm{rel,lab}}]^{\mathrm{vir}}
 =A_{\boldsymbol\Lambda}
 [\overline{\mathcal M}^{\mathrm{rel,part}}]^{\mathrm{vir}}.
\]
There is no additional factor $\prod e_{ij}$ in
\eqref{eq:labeled-relative-conversion}; contact-order products enter the
gluing factors in degeneration formulas, not the passage from labeled
roots to the partition convention.  Since
$A_{\boldsymbol\Lambda}>0$, nonvanishing is independent of the labeling
convention.

Theorem~\ref{thm:relative-GW-nonvanishing} uses a comparison theorem, not
only nonemptiness of the open Hurwitz locus.  Nodal sources, contracted
components, expanded targets, and rubber components can occur on the
boundary of the relative compactification.  The theorem does not assert
that this boundary is empty, that its individual contributions vanish, or
that the virtual class equals the ordinary fundamental class of the
Hurwitz-space closure.  It asserts the equality of their total virtual
degree with the weighted Hurwitz count.

\subsection{Completed cycles and a positive disconnected stationary sector}

Let $\omega\in H^2(\mathbf P^1,\mathbf Q)$ be the point class.  For
$\lambda,\mu\vdash p$, write
\[
 \left\langle\lambda\ \middle|\
   \prod_{i=1}^n\tau_{r_i-1}(\omega)\
   \middle|\ \mu\right\rangle^{\bullet,\mathbf P^1}_{g,p}
\]
for the possibly disconnected degree-$p$ stationary invariant relative to
$0$ and $\infty$, with relative profiles $\lambda$ and $\mu$.  Here $g$ is
the arithmetic genus: if the source has connected components of genera
$g_1,\ldots,g_c$, then
\[
 g=\sum_{j=1}^c g_j-c+1,
\]
so $g$ may be negative.  For partitions
$\Theta_1,\ldots,\Theta_s\vdash p$ and an integer $h$, define
\begin{equation}\label{eq:disconnected-hurwitz-number}
 H^\bullet_{h,p}(\Theta_1,\ldots,\Theta_s)
 =\frac1{p!}\#\left\{(\sigma_1,\ldots,\sigma_s):
  \sigma_i\in C_{\Theta_i},\quad
  \sigma_1\cdots\sigma_s=1\right\},
\end{equation}
provided
\begin{equation}\label{eq:disconnected-hurwitz-RH}
 \sum_{i=1}^s\bigl(p-\ell(\Theta_i)\bigr)=2p-2+2h;
\end{equation}
otherwise set it equal to zero.  No transitivity condition is imposed.
If the associated cover has connected components $C_1,\ldots,C_c$, then
$h=\sum_jg(C_j)-c+1$, so negative $h$ is allowed.  Every number in
\eqref{eq:disconnected-hurwitz-number} is nonnegative and, for $h\geq0$,
\begin{equation}\label{eq:connected-contained-in-disconnected}
 H^\bullet_{h,p}(\Theta_1,\ldots,\Theta_s)
 \geq H^\circ_{h,p}(\Theta_1,\ldots,\Theta_s).
\end{equation}

We recall the exact form of the completed-cycle expansion needed for the
positivity argument.  For a partition $\eta$, possibly empty, put
\[
 \widetilde\eta^{\,(p)}=\eta\cup1^{p-|\eta|},\qquad
 B_p(\eta)=
 \binom{m_1(\eta)+p-|\eta|}{m_1(\eta)}
 \quad (|\eta|\leq p).
\]
Put $B_p(\eta)=0$ when $|\eta|>p$.  In particular,
$B_p(\varnothing)=1$.
The factor $B_p(\eta)$ is the multiplicity in the extended Hurwitz
convention of \cite[equation~(0.5)]{OkounkovPandharipande}.  Define
\[
 \mathcal S(z)=\frac{\sinh(z/2)}{z/2}.
\]
The completed $r$-cycle is
\begin{equation}\label{eq:completed-cycle-expansion}
 \overline{(r)}=\sum_\eta\rho_{r,\eta}(\eta),
\end{equation}
where
\begin{equation}\label{eq:completion-coefficient}
 \rho_{r,\eta}
 =(r-1)!\frac{\prod_j\eta_j}{|\eta|!}
 [z^{\,r+1-|\eta|-\ell(\eta)}]
 \mathcal S(z)^{|\eta|-1}\prod_j\mathcal S(\eta_jz).
\end{equation}
For the empty partition the empty products are one, so
\begin{equation}\label{eq:empty-completion-coefficient}
 \rho_{r,\varnothing}
 =(r-1)![z^{r+1}]\frac1{\mathcal S(z)}.
\end{equation}
These are equations (0.21)--(0.22) of
\cite{OkounkovPandharipande}.  Since $\mathcal S(z)$ has nonnegative
coefficients, \eqref{eq:completion-coefficient} implies
\begin{equation}\label{eq:completion-positivity}
 \rho_{r,\eta}\geq0\quad(\eta\neq\varnothing),
 \qquad \rho_{r,(r)}=1.
\end{equation}
Moreover, $1/\mathcal S(z)$ is an even series.  Hence
\begin{equation}\label{eq:even-empty-vanishing}
 r\text{ even}\quad\Longrightarrow\quad
 \rho_{r,\varnothing}=0.
\end{equation}

For completeness, we spell out the arithmetic-genus grading in the
disconnected correspondence.  If $\rho_{r_i,\eta_i}\neq0$, set
\begin{equation}\label{eq:completion-genus-loss}
 a_i=\frac{r_i+1-|\eta_i|-\ell(\eta_i)}2,
 \qquad
 h(\boldsymbol\eta)
 =g-\sum_{i=1}^n\bigl(a_i+\ell(\eta_i)-1\bigr).
\end{equation}
The vanishing condition implicit in
\eqref{eq:completion-coefficient} makes $a_i$ a nonnegative integer.  The
integer $h(\boldsymbol\eta)$ need not be nonnegative because it is the
arithmetic genus of a possibly disconnected source.  A direct calculation
gives
\begin{equation}\label{eq:completion-defect-identity}
 p-\ell(\widetilde\eta_i^{\,(p)})
 =|\eta_i|-\ell(\eta_i)
 =r_i-1-2\bigl(a_i+\ell(\eta_i)-1\bigr).
\end{equation}
Thus a term indexed by $\boldsymbol\eta$ has covering genus
$h(\boldsymbol\eta)$.

\begin{proposition}[Disconnected completed-cycle expansion]
\label{prop:disconnected-completed-cycle-expansion}
Suppose
\begin{equation}\label{eq:relative-stationary-dimension}
 (p-\ell(\lambda))+(p-\ell(\mu))
 +\sum_{i=1}^n(r_i-1)=2p-2+2g.
\end{equation}
Then the disconnected relative Gromov--Witten/Hurwitz correspondence is
\begin{align}
&\left\langle\lambda\ \middle|\
   \prod_{i=1}^n\tau_{r_i-1}(\omega)\
   \middle|\ \mu\right\rangle^{\bullet,\mathbf P^1}_{g,p}
\label{eq:disconnected-GWH-expanded}\\
&\quad=
 \frac1{\prod_{i=1}^n(r_i-1)!}
 \sum_{\substack{\eta_1,\ldots,\eta_n\\|\eta_i|\leq p}}
 \left(\prod_{i=1}^n\rho_{r_i,\eta_i}B_p(\eta_i)\right)
 H^\bullet_{h(\boldsymbol\eta),p}
 \left(\lambda,\mu,
       \widetilde\eta_1^{\,(p)},\ldots,
       \widetilde\eta_n^{\,(p)}\right).
\nonumber
\end{align}
An identity profile $1^p$ in the Hurwitz number is a dummy unramified
condition and is deleted.  Equation
\eqref{eq:completion-defect-identity} verifies directly that each ordinary
Hurwitz number on the right has the arithmetic genus displayed in
\eqref{eq:disconnected-GWH-expanded}.  It also prevents the
completed-cycle genus grading from being confused with the genus of every
individual ordinary-cover term.
\end{proposition}

\begin{proof}
Theorem~1 of \cite{OkounkovPandharipande} gives
\[
 \left\langle\lambda\ \middle|\
   \prod_i\tau_{r_i-1}(\omega)\
   \middle|\ \mu\right\rangle^{\bullet,\mathbf P^1}_{p}
 =\frac1{\prod_i(r_i-1)!}
 H^{\bullet,\mathbf P^1}_{p}
 \bigl(\lambda,\mu,
       \overline{(r_1)},\ldots,\overline{(r_n)}\bigr).
\]
Expanding every completed cycle by
\eqref{eq:completed-cycle-expansion} gives a multilinear sum of ordinary
ramification conditions.  Equation~(0.5) of the cited paper replaces a
condition $\eta_i$ of size at most $p$ by
$\widetilde\eta_i^{\,(p)}$ and multiplies its ordinary Hurwitz number by
$B_p(\eta_i)$.  Conditions of size greater than $p$ vanish.  This proves
all coefficients in \eqref{eq:disconnected-GWH-expanded}.

It remains to identify the arithmetic genus.  Riemann--Hurwitz for a
possibly disconnected degree-$p$ cover gives
\[
 2h-2+2p
 =(p-\ell(\lambda))+(p-\ell(\mu))
  +\sum_i\bigl(|\eta_i|-\ell(\eta_i)\bigr).
\]
Subtract this equality from
\eqref{eq:relative-stationary-dimension} and use
\eqref{eq:completion-defect-identity}.  The result is
$h=h(\boldsymbol\eta)$.  Because the source may be disconnected, this
arithmetic genus can be negative and no such term may be discarded.
\end{proof}

\begin{theorem}[Disconnected positivity for even completed cycles]
\label{thm:GW-positive-sector}
Let $p$ be prime, let $g\geq0$, let $\lambda,\mu\vdash p$, and let
$r_1,\ldots,r_n$ be even integers with $2\leq r_i\leq p$.  Suppose
\eqref{eq:relative-stationary-dimension} holds.  Then
\begin{equation}\label{eq:GW-strict-positive}
 \left\langle\lambda\ \middle|\
   \prod_{i=1}^n\tau_{r_i-1}(\omega)\
   \middle|\ \mu\right\rangle^{\bullet,\mathbf P^1}_{g,p}>0.
\end{equation}
More precisely,
\begin{align}
&\left\langle\lambda\ \middle|\
   \prod_{i=1}^n\tau_{r_i-1}(\omega)\
   \middle|\ \mu\right\rangle^{\bullet,\mathbf P^1}_{g,p}
\label{eq:GW-leading-lower-bound}\\
&\quad\geq
 \frac1{\prod_i(r_i-1)!}
 H^\circ_{g,p}
 \left(\lambda,\mu,
 (r_1,1^{p-r_1}),\ldots,(r_n,1^{p-r_n})\right)>0,
\nonumber
\end{align}
where identity profiles on the right are deleted.
\end{theorem}

\begin{proof}
Condition \eqref{eq:relative-stationary-dimension} is Riemann--Hurwitz for
the principal passport
\begin{equation}\label{eq:principal-GW-passport}
 \lambda,\quad\mu,\quad
 (r_1,1^{p-r_1}),\ldots,(r_n,1^{p-r_n}).
\end{equation}
Indeed, the defect of $(r_i,1^{p-r_i})$ is $r_i-1$.  Delete $\lambda$ or
$\mu$ from the list if it is $1^p$.  The remaining profiles are a compatible
prime-degree branch datum, and Theorem~\ref{thm:main} gives
\begin{equation}\label{eq:principal-Hurwitz-positive}
 H^\circ_{g,p}
 \left(\lambda,\mu,
 (r_1,1^{p-r_1}),\ldots,(r_n,1^{p-r_n})\right)>0.
\end{equation}

Expand the left side of \eqref{eq:GW-strict-positive} by
\eqref{eq:disconnected-GWH-expanded}.  Because every $r_i$ is even,
\eqref{eq:even-empty-vanishing} removes every term in which some
$\eta_i$ is empty.  All remaining completion coefficients are nonnegative
by \eqref{eq:completion-positivity}; the factors $B_p(\eta_i)$ are positive
integers, and ordinary disconnected Hurwitz numbers are nonnegative
groupoid cardinalities.  Consequently every summand is nonnegative.

Choose $\eta_i=(r_i)$ for every $i$.  Then $a_i=0$,
$\ell(\eta_i)=1$, $h(\boldsymbol\eta)=g$,
$B_p((r_i))=1$, and $\rho_{r_i,(r_i)}=1$.  The corresponding summand is the
disconnected Hurwitz number for the principal passport divided by
$\prod_i(r_i-1)!$.  It is at least the positive connected contribution in
\eqref{eq:principal-Hurwitz-positive}, by
\eqref{eq:connected-contained-in-disconnected}.  This proves
\eqref{eq:GW-leading-lower-bound}, and hence strict positivity.
\end{proof}

Proposition~\ref{prop:disconnected-completed-cycle-expansion} is not a
termwise formula for connected invariants.  Taking the logarithm of the
full disconnected partition function produces connected cumulants.  A
correction profile $\eta_i$ can meet several connected components of an
ordinary Hurwitz cover, while the corresponding relative component can
join those components.  Thus the logarithm is a sum over connected
incidence graphs; it cannot be obtained by replacing $H^\bullet$ by
$H^\circ$ in \eqref{eq:disconnected-GWH-expanded}.

The case of simple stationary insertions contains no completion correction
at all.

\begin{corollary}[The transposition sector]
\label{cor:tau-one-positive}
Let $p$ be prime, let $g\geq0$, and let $\lambda,\mu\vdash p$.  If
\[
 N=2g-2+\ell(\lambda)+\ell(\mu)\geq0,
\]
then
\begin{equation}\label{eq:tau-one-Hurwitz}
 \left\langle\lambda\ \middle|\
   \tau_1(\omega)^N\
   \middle|\ \mu\right\rangle^{\circ,\mathbf P^1}_{g,p}
 =H^\circ_{g,p}
 \left(\lambda,\mu,(2,1^{p-2})^N\right)>0.
\end{equation}
\end{corollary}

\begin{proof}
Here \eqref{eq:relative-stationary-dimension} is exactly the definition of
$N$.  Okounkov--Pandharipande note that the transposition is the unique
cycle requiring no completion: $\overline{(2)}=(2)$.  We spell out the
connected extraction so that no termwise connected completed-cycle formula
is being used.  Introduce variables
\[
 q,\quad x,\quad t_1,t_2,\ldots,\quad s_1,s_2,\ldots,
\]
and write $t_\alpha=\prod_jt_{\alpha_j}$ and
$s_\beta=\prod_js_{\beta_j}$, and form
\begin{align*}
 Z^\bullet_{\mathrm{GW}}
 &=1+\sum_{d>0}q^d\sum_{\alpha,\beta\vdash d}
   t_\alpha s_\beta\sum_{m\geq0}\frac{x^m}{m!}
   \left\langle\alpha\ \middle|\tau_1(\omega)^m
   \middle|\beta\right\rangle^{\bullet,\mathbf P^1}_{d},\\
 Z^\bullet_{\mathrm H}
 &=1+\sum_{d>0}q^d\sum_{\alpha,\beta\vdash d}
   t_\alpha s_\beta\sum_{m\geq0}\frac{x^m}{m!}
   H^{\bullet,\mathrm{ext}}_d\bigl(\alpha,\beta,(2)^m\bigr).
\end{align*}
Here the arithmetic genus in every coefficient is the one forced by
Riemann--Hurwitz.  The superscript $\mathrm{ext}$ denotes the extended
Hurwitz convention of equation~(0.5) in the cited paper: in degree
$d\geq2$, the condition $(2)$ is padded to $(2,1^{d-2})$, while a term with
$d<2$ and $m>0$ vanishes.  Theorem~1 of
\cite{OkounkovPandharipande}, together with
$\overline{(2)}=(2)$, gives
$Z^\bullet_{\mathrm{GW}}=Z^\bullet_{\mathrm H}$.

Multiplication of $t_\alpha$ and $s_\beta$ concatenates the profiles under
disjoint union.  The stack weights give the exponential formula, so
$\log Z^\bullet_{\mathrm{GW}}$ generates connected stable maps, whereas
$\log Z^\bullet_{\mathrm H}$ generates connected, equivalently transitive,
Hurwitz covers; compare equations~(4.1)--(4.2) of the cited paper.  Equality
of the coefficients of $q^pt_\lambda s_\mu x^N/N!$ proves the equality in
\eqref{eq:tau-one-Hurwitz}.  The definition of $N$ is Riemann--Hurwitz for
that passport, and Theorem~\ref{thm:main} gives strict positivity after
identity profiles are deleted.
\end{proof}

\begin{example}\label{ex:degree-seven-GW}
Take
\[
 p=7,\qquad g=1,\qquad
 \lambda=(6,1),\qquad \mu=(4,3),\qquad (r_1,r_2)=(4,2).
\]
The defects satisfy
\[
 (7-2)+(7-2)+(4-1)+(2-1)=5+5+3+1=14=2\cdot7-2+2.
\]
Theorem~\ref{thm:GW-positive-sector} gives
\[
 \left\langle(6,1)\ \middle|\
  \tau_3(\omega)\tau_1(\omega)
  \middle|\ (4,3)\right\rangle^{\bullet,\mathbf P^1}_{1,7}>0.
\]
More precisely, the principal term alone gives the lower bound
\[
 \left\langle(6,1)\ \middle|\
  \tau_3(\omega)\tau_1(\omega)
  \middle|\ (4,3)\right\rangle^{\bullet,\mathbf P^1}_{1,7}
 \geq\frac1{3!}
 H^\circ_{1,7}\bigl((6,1),(4,3),(4,1^3),(2,1^5)\bigr)>0.
\]
\end{example}

Finally, define the relative partition function
\begin{equation}\label{eq:relative-GW-partition-function}
 \tau_{\mathbf P^1}(x,t,s)
 =\sum_{|\lambda|=|\mu|}t_\lambda s_\mu
 \left\langle\lambda\ \middle|\
 \exp\!\left(\sum_{j\geq0}x_j\tau_j(\omega)\right)
 \middle|\ \mu\right\rangle^{\bullet,\mathbf P^1},
\end{equation}
where $t_\lambda=\prod_i t_{\lambda_i}$ and similarly for $s_\mu$.
Okounkov--Pandharipande prove that $\tau_{\mathbf P^1}(x,t,s)$ is the
$n=0$ member of a sequence
$\{\tau_{\mathbf P^1}(x,t,s,n)\}_{n\in\mathbf Z}$ forming a $2$-Toda
tau-function, and that $\log\tau_{\mathbf P^1}$ generates the connected invariants
\cite[Section~4, especially equations~(4.1)--(4.2) and Theorem~4]
{OkounkovPandharipande}.  Theorem~\ref{thm:GW-positive-sector} therefore
identifies a strictly positive family of coefficients in the degree-$p$
part of its $n=0$ member: the two relative states are
arbitrary partitions of $p$, and every stationary index $j=r_i-1$ is odd
with $r_i\leq p$, subject only to
\eqref{eq:relative-stationary-dimension}.  Corollary~\ref{cor:tau-one-positive}
gives, in addition, a strictly positive degree-$p$ family in the connected
free energy $\log\tau_{\mathbf P^1}$ along the transposition direction
$x_1$.

The parity hypothesis in Theorem~\ref{thm:GW-positive-sector} is essential
for this argument.  For an odd completed cycle, the coefficient of the
empty partition in \eqref{eq:empty-completion-coefficient} need not vanish
and can have either sign.  Thus Corollary~\ref{cor:hurwitz-support} does not
by itself imply nonvanishing of every stationary Gromov--Witten invariant.
It gives full support in the ordinary ramification basis, while
Theorem~\ref{thm:GW-positive-sector} is the sector in the completed-cycle
basis where all boundary corrections have a controlled sign.\\

\noindent{\bf Acknowledgments.}
The authors acknowledge ChatGPT's assistance in the initial exploration of the proof of Theorem
1.1. Inspired by interactions with the authors in September 2026, they rigorously developed and verified the
argument, assuming full responsibility for the paper.

J.S.  is especially grateful to Prof. Bin Xu, who first introduced him to the prime-degree conjecture in 2016, when he was a Ph.D. student. Since then, he has continued to think about the problem and possible approaches to its resolution. He also thanks Michael Zieve for a very helpful discussion of the problem during Zieve’s visit to the University of Science and Technology of China. J.S. further extends his sincere thanks to Qing Chen, Weibo Fu, Mao Sheng, and Xiaotao Sun for valuable conversations throughout the extended gestation period of this work.

\sloppy

\vspace{0.5cm}

{\small {\sc  \noindent Jijian Song\\
Center for Applied Mathematics and KL-AAGDM\\
Tianjin University\\
Tianjin 300072 China}\\
jijian.song@tju.edu.cn\\

{\sc \noindent Hailin Wen\\
Wu Wen-Tsun Key Laboratory of Math, USTC, CAS\\
School of Mathematical Sciences\\
University of Science and Technology of China\\
Hefei 230026 China}\\
wenhailin@mail.ustc.edu.cn\\

{\sc \noindent Zebao Zhang\\
Mathematical Science Research Center\\
Chongqing University of Technology\\
Chongqing 400054 China}\\
zhangzebao@cqut.edu.cn}\\

\fussy

\end{document}